\documentclass[final,onefignum,onetabnum]{siamart190516}

\usepackage{amsmath}
\usepackage{amsfonts}
\usepackage{amssymb}
\usepackage{latexsym}
\usepackage[margin=1in]{geometry}
\usepackage{graphicx}
\usepackage{xcolor}
\usepackage{epstopdf}
\usepackage{tikz}

\numberwithin{equation}{section}

\def\curl{\operatorname{curl}} 
\def\i{\operatorname{\mathbf{i}}} 
\def\H{\mathsf{H}} 
\def\TH{\mathsf{TH}} 
\def\Vol{\operatorname{Vol}} 

\ifpdf
  \DeclareGraphicsExtensions{.eps,.pdf,.png,.jpg}
\else
  \DeclareGraphicsExtensions{.eps}
\fi

\newsiamremark{remark}{Remark}
\newsiamremark{hypothesis}{Hypothesis}
\crefname{hypothesis}{Hypothesis}{Hypotheses}
\newsiamthm{claim}{Claim}

\headers{A linearised inverse conductivity problem}
{V. Isakov, S. Lu, and B. Xu}

\title{A linearised inverse conductivity problem for the Maxwell system at a high frequency
\thanks{Submitted to the editors DATE.
\funding{
V. Isakov is supported in part by the Emylou Keith and Betty Dutcher Distinguished Professorship and the NSF grants DMS 15-14886 and DMS 20-08154. 
S. Lu is supported by NSFC (No.11925104), Program of Shanghai Academic/Technology Research Leader (19XD1420500) and National Key Research and Development Program of China (No. 2017YFC1404103). 
B. Xu is supported by NSFC (No.11801351) and the Shanghai Pujiang Program (18PJ1403600).}}}

\author{
Victor Isakov
\thanks{Department of Mathematics, Statistics, and Physics, Wichita State University, Wichita, KS 67260-0033, USA 
(\email{victor.isakov@wichita.edu}).}
\and
Shuai Lu
\thanks{LMNS, SKLCAM, School of Mathematical Sciences, Fudan University, No.220 Road Handan, Shanghai 200433, China 
(\email{slu@fudan.edu.cn}).}
\and
Boxi Xu
\thanks{Corresponding author. School of Mathematics, Shanghai University of Finance and Economics, No.777 Road Guoding, Shanghai 200433, China 
(\email{xu.boxi@mail.sufe.edu.cn}).}
}

\usepackage{amsopn}

\ifpdf
\hypersetup{
  pdftitle={A linearised inverse conductivity problem for the Maxwell system at a high frequency},
  pdfauthor={V. Isakov, S. Lu, and B. Xu}
}
\fi

\begin{document}

\maketitle

\begin{abstract}
We consider a linearised inverse conductivity problem for electromagnetic waves in a three dimensional bounded domain at a high time-harmonic frequency. Increasing stability bounds for the conductivity coefficient in the full Maxwell system and in a simplified transverse electric mode are derived. These bounds contain a Lipschitz term with a factor growing polynomially in terms of the frequency, a H\"{o}lder term, and a logarithmic term which decays with respect to the frequency as a power. To validate this increasing stability numerically, we propose a reconstruction algorithm aiming at the recovery of sufficiently many Fourier modes of the conductivity. A numerical evidence sheds light on the influence of the growing frequency and confirms the improved resolution at higher frequencies.
\end{abstract}




\section{Introduction}\label{se_intro}

The stationary electromagnetic field $(E,H)$ satisfies the Maxwell system
\begin{align}
\label{system1}
\curl E - \i \omega \mu_{0} H = 0, \quad
\curl H +\i \omega \epsilon_{0} E = \sigma E,
\end{align}
in the medium $\Omega \subset \mathbb{R}^{3}$ with the electric permittivity $\epsilon_{0}$, the magnetic permeability $\mu_{0}$ and the conductivity coefficient $\sigma \in L^{\infty}(\Omega)$.
We aim at recovery of the conductivity $\sigma$ from measurements of all possible electromagnetic boundary data or the boundary mapping
\begin{align}
\label{boundarymap}
\Lambda: E \times \nu |_{\partial\Omega} \mapsto H \times \nu |_{\partial\Omega},
\end{align}
when the time-harmonic frequency $\omega$ of stationary waves is assumed to be large and both $\epsilon_{0}$, $\mu_{0}$ are assumed to be constants. Since such an inverse boundary value problem is exponentially ill-posed, highly non-linear and moreover non-convex, we shall simplify the original inverse problem into a linearised one, which is numerically feasible.

Investigation of the inverse conductivity problem goes back to 1980s when the elliptic equation arising in the electrical impedance tomography was transformed into a Schr\"{o}dinger equation
\begin{align*}
(-\Delta + c) u = 0 \quad \textrm{in\ } \Omega.
\end{align*}
The global uniqueness of the Schr\"{o}dinger potential $c(x)$ from the Dirichlet-to-Neumann map in three and higher dimensions was derived by constructing complex geometrical optics solutions in \cite{SU1987}. Later on, the stability estimate for this inverse problem is demonstrated to be of a logarithmic type in \cite{A1988}, \cite{M2001}. Recently, a series of papers showed that the stability estimate improves when one considers the Schr\"{o}dinger equation
\begin{align*}
(-\Delta - k^{2} + \i k b + c) u = 0 \quad \textrm{in\ } \Omega
\end{align*}
with a large wave number $k$, an attenuation constant $b$ and a potential function $c(x)$. If $b = 0$, one obtained an increasing stability estimate in the inverse Schr\"{o}dinger potential problem within different ranges of the wave numbers in \cite{I2011}. If the constant $b$ is greater than $0$, then the results in \cite{IW2014} show that the increasing stability involves a linearly exponential dependence on the attenuation constant. We note that the above mentioned increasing stability holds true in three and higher dimensions under different a priori regularity assumptions. To numerically reconstruct the potential function $c(x)$ in two and higher dimensions the authors of current work have considered a linearised inverse Schr\"{o}dinger potential problem in \cite{ILX2018} whose increasing stability is also derived theoretically and verified numerically.

In \cite{SIC1992}, a linearised problem of identification of $\epsilon_0$, $\mu_0$, $\sigma$ in the Maxwell system \eqref{system1} from all possible boundary measurements was proposed. The first global uniqueness result for all these electromagnetic parameters of an isotropic medium is obtained in \cite{OPS1993} by the admittance or impedance mappings on the boundary. Meanwhile, the uniqueness and a logarithmic stability estimate hold true if one considers a Cauchy data set on the full or partial boundary as shown in \cite{C2010,C2011}. An increasing stability estimate of the conductivity $\sigma$ in the Maxwell equation \eqref{system1} is proven in \cite{ILW2016} where the Cauchy data set on the full boundary is used. Despite the above theoretical results, numerical schemes reconstructing the electromagnetic parameters are not well developed.

Out of two features of the inverse conductivity problem, an exponential instability seems to be more difficult for an analysis and more damaging for numerical reconstruction. Since non-convexity produces additional difficulties, in our opinion linearisation should help to understand better stability for conductivity function $\sigma$. We note that a linearisation approach has been proposed in \cite{SIC1992}, but no numerical examples are reported there. In the current work we consider a linearised inverse conductivity problem for the Maxwell system \eqref{system1} and investigate the increasing stability of the linearised inverse problem with respect to the growing frequency $\omega$. Moreover, we propose a Fourier based reconstruction algorithm aiming at recovering the dominating Fourier modes of the conductivity $\sigma$ for higher frequencies $\omega$. We shall mention that numerical algorithms on reconstructing the non-constant medium coefficient of the Maxwell system have been proposed in \cite{BL2005,BL2009} where varying polarized or plane incident waves are used at a fixed frequency.  We take a similar approach for the inverse conductivity problem of the Maxwell system \eqref{system1} in the current work.

The paper is organized as follows. The main increasing stability estimates for linearised inverse conductivity problems of the full Maxwell system and a simplified Transverse Electric (TE) mode are derived in Section \ref{se_main}. To obtain these estimates, we construct appropriate complex exponential (CE) solutions for the Maxwell system with the constant electric permittivity $\epsilon_{0}$ and magnetic permeability $\mu_{0}$. By using these CE solutions, increasing stability estimates are derived which contain a Lipschitz part with the factor growing polynomially in $\omega$, a H\"{o}lder part, and a logarithmic part which polynomially decays in $\omega$. These estimates are explicit, i.e. which do not contain unknown constants, like in \cite{ILW2016}. A Fourier-based reconstruction algorithm is introduced in Section \ref{se_numer} aiming at the recovery of sufficiently many Fourier modes of the unknown conductivity function $\sigma$ for a high frequency $\omega$. Numerical examples confirm the efficiency of the proposed algorithm and numerically verify the improving resolution of the reconstructed conductivity when the frequency grows. Moreover, in numerical solution of the linearised inverse problem the data from the original inverse problem are used, and for higher frequencies a very good approximation of $\sigma$ is obtained.


\section{Stability estimates}\label{se_main}

\subsection{Setups}\label{se_main_basic}

In the Maxwell system \eqref{system1} we assume that both electrical permittivity $\epsilon_{0}$ and magnetic permeability $\mu_{0}$ are positive constants and the conductivity $\sigma \in L^{\infty}(\Omega)$. The wave number $k$ is defined by $k := \omega \sqrt{\epsilon_{0} \mu_{0}}$.

The domain $\Omega$ is assumed to be bounded with the boundary $\partial\Omega \in C^{2}$. $\H^{\ell}(\Omega)$ denotes the standard Sobolev space with the norm $\|\cdot\|_{(\ell)}(\Omega)$. The function space $\H(\Omega;\curl)$ is defined by $\H(\Omega;\curl) = \{ \mathbf{u} : \mathbf{u} \in \H^{0}(\Omega), \curl \mathbf{u} \in \H^{0}(\Omega) \}$ with the standard norm. We recall that $\TH(\partial\Omega)$ is the space of traces of functions from $\H(\Omega;\curl)$ with the norm $\|\mathbf{g}\|_{\TH(\partial\Omega)} = \inf\big( \|\mathbf{u}\|_{(0)}(\Omega) + \|\curl\mathbf{u}\|_{(0)}(\Omega) \big)$ over all $\mathbf{u} \in \H(\Omega;\curl)$ such that $\mathbf{u} \times \nu = \mathbf{g}$ on $\partial\Omega$ and $\nu$ is the outward unit normal vector of $\partial\Omega$.

Under above assumptions and notations, it is known that there is a unique solution $(E,H) \in (\H(\Omega;\curl))^{2}$ of \eqref{system1} and the boundary condition is
\begin{align*}
E \times \nu = \mathbf{g} \quad \textrm{on\ } \partial\Omega,
\end{align*}
provided that $\mathbf{g} \in \TH(\partial\Omega)$ and $\omega$ is not an eigenvalue.

We expect that if the conductivity $\sigma$ is small or the frequency $\omega$ is sufficiently large, the solution $(E,H)$ of \eqref{system1} has the form
\begin{align*}
E = E( ;0) + E( ;1) + E( ;2), \quad
H = H( ;0) + H( ;1) + H( ;2),
\end{align*}
where $(E( ;\ell), H( ;\ell))$, $\ell = 0,1$ are solutions of the boundary value problems below and $(E( ;2),H( ;2))$ are higher-order terms with respect to the conductivity $\sigma$. The solution $(E( ;0),H( ;0))$ satisfies the unperturbed problem
\begin{align}
\label{system2}
\curl E( ;0) - \i \omega \mu_{0} H( ;0) = 0, \quad
\curl H( ;0) + \i \omega \epsilon_{0} E( ;0) = 0 \quad \textrm{in\ } \Omega,
\end{align}
and the boundary condition $E( ;0) \times \nu = \mathbf{g}$ on $\partial\Omega$.
Meanwhile, the solution $(E( ;1),H( ;1))$ satisfies
\begin{align}
\label{system3}
\curl E( ;1) - \i \omega \mu_{0} H( ;1) = 0, \quad
\curl H( ;1) + \i \omega \epsilon_{0} E( ;1) = \sigma E( ;0) \quad \textrm{in\ } \Omega,
\end{align}
and the boundary condition
\begin{align}
\label{boundary3}
E( ;1) \times \nu = 0 \quad \textrm{on\ } \partial\Omega.
\end{align}
The remaining higher-order terms would be small if the conductivity is small, for instance, $\|E( ;2)\|_{(0)}(\Omega) \leq C \|\sigma\|^{1+\eta}_{L^{\infty}(\Omega)}$ in \cite{SIC1992} for some $\eta \in (0,1)$.

Similarly to \cite{SIC1992}, the linearised inverse problem to be considered in current work is to {\bf recover the conductivity $\sigma$ given the linearised boundary mapping}
\begin{align}
\label{linearmap}
\Lambda' : E( ;0) \times \nu |_{\partial\Omega} \mapsto H( ;1) \times \nu |_{\partial\Omega}.
\end{align}

To proceed further we introduce an "adjoint" solution $(E^{*}( ;0), H^{*}( ;0)) \in (\H(\Omega;\curl))^{2}$ of the unperturbed system \eqref{system2}, i.e.,
\begin{align}
\label{system2*}
\curl E^{*}( ;0) - \i \omega \mu_{0} H^{*}( ;0) = 0, \quad
\curl H^{*}( ;0) + \i \omega \epsilon_{0} E^{*}( ;0) = 0 \quad \textrm{in\ } \Omega.
\end{align}

By the well-known identity
\begin{align}
\label{equality}
\int_{\Omega} ( \mathbf{v} \cdot \curl \mathbf{w} - \curl \mathbf{v} \cdot \mathbf{w} ) \,\mathrm{d}x \,
= \int_{\partial\Omega} ( \mathbf{v} \times \nu ) \cdot \mathbf{w} \,\mathrm{d}S,
\end{align}
with $\mathbf{v} = H( ;1)$, $\mathbf{w} = E^{*}( ;0)$ and $\mathbf{v} = E( ;1)$, $\mathbf{w} = H^{*}( ;0)$ respectively, we obtain
\begin{align*}
\begin{aligned}
&\int_{\Omega} \big( H( ;1) \cdot \curl E^{*}( ;0) - \curl H( ;1) \cdot E^{*}( ;0)
+ E( ;1) \cdot \curl H^{*}( ;0) - \curl E( ;1) \cdot H^{*}( ;0) \big) \,\mathrm{d}x \\
&= \int_{\partial\Omega} \big( ( H( ;1) \times \nu ) \cdot E^{*}( ;0)
+ ( E( ;1) \times \nu ) \cdot H^{*}( ;0) \big) \,\mathrm{d}S.
\end{aligned}
\end{align*}
By using \eqref{system3}, \eqref{boundary3}, \eqref{system2*} the above equality further reduces to
\begin{align}
\label{hatsigma}
\int_{\Omega} \sigma E( ;0) \cdot E^{*}( ;0) \,\mathrm{d}x \,
= - \int_{\partial\Omega} ( H( ;1) \times \nu ) \cdot E^{*}( ;0) \,\mathrm{d}S
\end{align}
which plays a fundamental role in this work.


\subsection{Increasing stability for the linearised inverse conductivity problem}\label{se_main_stability}

Before we proceed further, some additional notations are introduced. Here and in what follows, we let $\varepsilon > 0$ be the operator norm of the linearised operator $\Lambda'$ in \eqref{linearmap} from $\TH(\partial\Omega)$ to $\TH(\partial\Omega)$, $D = \sup |x-y|$ with $x,y \in \Omega$ be the diameter of $\Omega$, $\Vol(\Omega)$ be the volume of $\Omega$ and $\Vol_{2}(\Omega) = \sup \Vol_{2}(\Omega')$ over all $2$-dimensional orthogonal projection $\Omega'$ of $\Omega$. We will extend $\sigma$ onto $\mathbb{R}^{3} \setminus \Omega$ as zero and recall that the Fourier transform
\begin{equation*}
\widehat{\sigma}(\xi) = (2\pi)^{-\frac{3}{2}} \int_{\mathbb{R}^{3}} \sigma(x) \, e^{-\i \xi \cdot x} \,\mathrm{d}x.
\end{equation*}

The main increasing stability estimate for the linearised inverse conductivity problem is presented below.

\begin{theorem}
\label{thm1}
Let $\|\sigma\|_{(1)}(\Omega) \leq M_{1}$, $\sigma = 0$ on $\partial\Omega$, and the wave number $k \geq 1$. Furthermore, let $0 < \alpha \leq 1$ and $\varepsilon < 1$, then
\begin{equation}
\label{bound1}
\begin{aligned}
\|\sigma\|^{2}_{(0)}(\Omega)
&\leq \frac{4}{3\pi^{2}} (\Vol(\Omega))^{2} \Big( (1-\chi_{k}(\mathcal{E})) (1+k)^{4} k^{3\alpha}
+ \chi_{k}(\mathcal{E}) (1+\mathcal{E})^{4} \mathcal{E}^{3} \Big) \varepsilon^{2} \\
&\qquad + \chi_{k}(\mathcal{E}) \frac{1}{4\pi^{2}} (\Vol_{2}(\Omega))^{2} \frac{1}{D} \Big(1+(2D)^{2}\Big)^{1/2} (1+\mathcal{E})^{4} \mathcal{E} \varepsilon \\
&\qquad + M_{1}^{2} \Big( (1-\chi_{k}(\mathcal{E})) \frac{1}{1+2\mathcal{E}^{2\alpha}+2k^{2\alpha}}
+ \chi_{k}(\mathcal{E}) \frac{1}{1+\frac{\mathcal{E}^{2}}{D^{2}}+4k^{2}} \Big).
\end{aligned}
\end{equation}
Here $\mathcal{E} = -\ln\varepsilon$ and $\chi_{k}(\mathcal{E}) = 0$ if $\mathcal{E} < k$ and $\chi_{k}(\mathcal{E}) = 1$ if $k \leq \mathcal{E}$.
\end{theorem}

\begin{remark}
Since in the dimension three $\Vol(\Omega) \leq \frac{\pi D^{3}}{6}$, $\Vol_{2}(\Omega) \leq \frac{\pi D^{2}}{4}$, then the bound \eqref{bound1} implies that
\begin{equation*}
\begin{aligned}
\|\sigma\|^{2}_{(0)}(\Omega)
&\leq \frac{D^{6}}{3^{3}} \Big( (1-\chi_{k}(\mathcal{E})) (1+k)^{4} k^{3\alpha}
+ \chi_{k}(\mathcal{E}) (1+\mathcal{E})^{4} \mathcal{E}^{3} \Big) \varepsilon^{2} \\
&\qquad + \chi_{k}(\mathcal{E}) \frac{D^{3}}{2^{6}} \big(1+(2D)^{2}\big)^{1/2} (1+\mathcal{E})^{4} \mathcal{E} \varepsilon \\
&\qquad + M_{1}^{2} \Big( (1-\chi_{k}(\mathcal{E})) \frac{1}{1+2\mathcal{E}^{2\alpha}+2k^{2\alpha}}
+ \chi_{k}(\mathcal{E}) \frac{1}{1+\frac{\mathcal{E}^{2}}{D^{2}}+4k^{2}} \Big).
\end{aligned}
\end{equation*}
\end{remark}

\begin{proof}
Due to the translational invariance we can assume that the origin $0 \in \Omega$ and $D = 2\sup |x|$, $x \in \Omega$.

We construct CE solutions
\begin{equation}
\label{complexsol}
E(x ;0) = a e^{\i \zeta \cdot x}, \quad
H(x ;0) = b e^{\i \zeta \cdot x}, \quad
E^{*}(x ;0) = a^{*} e^{\i \zeta^{*} \cdot x}, \quad
H^{*}(x ;0) = b^{*} e^{\i \zeta^{*} \cdot x},
\end{equation}
with complex valued vectors $a$, $b$, $a^{*}$, $b^{*}$, $\zeta$, $\zeta^{*}$, which yields
\begin{equation}
\label{curl_complexsol}
\begin{aligned}
&\curl E(x ;0) = \i e^{\i \zeta \cdot x} \zeta \times a, &\quad
&\curl H(x ;0) = \i e^{\i \zeta \cdot x} \zeta \times b, \\
&\curl E^{*}(x ;0) = \i e^{\i \zeta^{*} \cdot x} \zeta^{*} \times a^{*}, &\quad
&\curl H^{*}(x ;0) = \i e^{\i \zeta^{*} \cdot x} \zeta^{*} \times b^{*}.
\end{aligned}
\end{equation}
Then these solutions in \eqref{complexsol} solve the unperturbed Maxwell systems \eqref{system2} and \eqref{system2*} if and only if
\begin{equation}
\label{ab}
\begin{aligned}
&\zeta \times a = \omega \mu_{0} b, \quad
\zeta \times b = -\omega \epsilon_{0} a, \quad
\zeta^{*} \times a^{*} = \omega \mu_{0} b^{*}, \quad
\zeta^{*} \times b^{*} = -\omega \epsilon_{0} a^{*}, \\
&\zeta \cdot \zeta = \zeta^{*} \cdot \zeta^{*} = k^{2} = \omega^{2} \epsilon_{0} \mu_{0}.
\end{aligned}
\end{equation}

Letting $\xi \in \mathbb{R}^{3} \setminus \{0\}$, and $\{ e_{1} = \frac{\xi}{|\xi|}, e_{2}, e_{3} \}$ be an orthonormal basis in $\mathbb{R}^{3}$, we choose
\begin{equation}
\label{ab_choice}
\begin{aligned}
&\zeta = \frac{|\xi|}{2} e_{1} + \sqrt{k^{2}-\frac{|\xi|^{2}}{4}} e_{2}, &\quad
&\zeta^{*} = \frac{|\xi|}{2} e_{1} - \sqrt{k^{2}-\frac{|\xi|^{2}}{4}} e_{2}, \\
&a = e_{3}, \quad
b = \frac{-1}{\omega \mu_{0}} \Big( \frac{|\xi|}{2} e_{2} - \sqrt{k^{2}-\frac{|\xi|^{2}}{4}} e_{1} \Big), &\quad
&a^{*} = -e_{3}, \quad
b^{*} = \frac{1}{\omega \mu_0} \Big( \frac{|\xi|}{2} e_{2} + \sqrt{k^{2}-\frac{|\xi|^{2}}{4}} e_{1} \Big),
\end{aligned}
\end{equation}
where $\sqrt{-y} = \i \sqrt{y}$ when $y > 0$. One can easily verify that \eqref{ab_choice} implies \eqref{ab}. Thus the CE solutions defined in \eqref{complexsol} and given the parameter set (\ref{ab_choice}) satisfy \eqref{system2} and \eqref{system2*}.

Recalling the equality \eqref{hatsigma} and inserting the CE solutions in \eqref{complexsol} with the parameter set in \eqref{ab_choice}, we yield
\begin{equation*}
(2\pi)^{\frac{3}{2}} \widehat{\sigma}(-\xi)
= - \int_{\Omega} \sigma E( ;0) \cdot E^{*}( ;0) \,\mathrm{d}x \,
= \int_{\partial\Omega} ( H( ;1) \times \nu ) \cdot E^{*}( ;0) \,\mathrm{d}S.
\end{equation*}
In the above equality, CE solutions $E( ;0)$, $E^{*}( ;0)$ solve \eqref{system2}, \eqref{system2*} and $H( ;1)$ satisfies \eqref{system3}-\eqref{boundary3}. Hence
\begin{equation*}
(2\pi)^{\frac{3}{2}} \left| \widehat{\sigma}(-\xi) \right|
= \left| \int_{\partial\Omega} ( H( ;1) \times \nu ) \cdot E^{*}( ;0) \,\mathrm{d}S \right|
= \left| \int_{\Omega} \big( H(\delta) \cdot \curl E^{*}( ;0) - \curl H(\delta) \cdot E^{*}( ;0) \big) \,\mathrm{d}x \, \right|
\end{equation*}
due to \eqref{equality}, provided $H(\delta) \times \nu = H( ;1) \times \nu$ on $\partial\Omega$.

By the definition of $\Lambda'$ in \eqref{linearmap}, we have
\begin{equation*}
\|H( ;1) \times \nu\|_{\TH(\partial\Omega)} \leq \varepsilon \|E( ;0) \times \nu\|_{\TH(\partial\Omega)}.
\end{equation*}
Using the definition of the norm in $\TH(\partial\Omega)$, for any $\delta > 0$, there exists $H(\delta)$ such that
\begin{equation}
\label{Hnorm}
\|H(\delta)\|_{(0)}(\Omega) + \|\curl H(\delta)\|_{(0)}(\Omega)
\leq \|H( ;1) \times \nu\|_{\TH(\partial\Omega)} + \delta
\leq \varepsilon \|E( ;0) \times \nu\|_{\TH(\partial\Omega)} + \delta.
\end{equation}
Now from \eqref{Hnorm}, we have
\begin{equation*}
\begin{aligned}
(2\pi)^{\frac{3}{2}} \left| \widehat{\sigma}(-\xi) \right|
&\leq \|H(\delta)\|_{(0)}(\Omega) \, \|\curl E^{*}( ;0)\|_{(0)}(\Omega) + \|\curl H(\delta)\|_{(0)}(\Omega) \, \|E^{*}( ;0)\|_{(0)}(\Omega) \\
&\leq \big( \varepsilon \|E( ;0) \times \nu\|_{\TH(\partial\Omega)} + \delta \big) \big( \|E^{*}( ;0)\|_{(0)}(\Omega) + \|\curl E^{*}( ;0)\|_{(0)}(\Omega) \big).
\end{aligned}
\end{equation*}
Since it holds for any positive $\delta$, by letting $\delta \to 0$ we derive
\begin{equation}
\label{hatsigmabound}
(2\pi)^{\frac{3}{2}} \left| \widehat{\sigma}(-\xi) \right|
\leq \varepsilon \|E( ;0) \times \nu\|_{\TH(\partial\Omega)} \big( \|E^{*}( ;0)\|_{(0)}(\Omega) + \|\curl E^{*}( ;0)\|_{(0)}(\Omega) \big).
\end{equation}

Using \eqref{complexsol}, \eqref{curl_complexsol} and \eqref{ab_choice}, we further obtain
\begin{equation}
\label{Enorm*}
\begin{aligned}
&\|E^{*}( ;0)\|^{2}_{(0)}(\Omega) \leq \Vol(\Omega),
& &\|\curl E^{*}( ;0)\|^{2}_{(0)}(\Omega) \leq k^{2} \Vol(\Omega),
& &\textrm{if\ } |\xi| \leq 2k, \\
&\|E^{*}( ;0)\|^{2}_{(0)}(\Omega) \leq \int_{\Omega} e^{\Xi\, e_{2} \cdot x} \,\mathrm{d}x,
& &\|\curl E^{*}( ;0)\|^{2}_{(0)}(\Omega) \leq k^{2} \int_{\Omega} e^{\Xi\, e_{2} \cdot x} \,\mathrm{d}x,
& &\textrm{if\ } 2k < |\xi|,
\end{aligned}
\end{equation}
where $\Xi = \sqrt{|\xi|^{2}-4k^{2}}$. Similarly,
\begin{equation}
\label{ETHnorm}
\begin{aligned}
&\|E( ;0) \times \nu\|_{\TH(\partial\Omega)}
\leq \|E( ;0)\|_{(0)}(\Omega) + \|\curl E( ;0)\|_{(0)}(\Omega)
\leq (1+k) (\Vol(\Omega))^{1/2},
& &\textrm{if\ } |\xi| \leq 2k, \\
&\|E( ;0) \times \nu\|_{\TH(\partial \Omega)}
\leq (1+k) \Big( \int_{\Omega} e^{-\Xi\, e_{2} \cdot x} \,\mathrm{d}x \Big)^{1/2},
& &\textrm{if\ } 2k < |\xi|.
\end{aligned}
\end{equation}

Now combining \eqref{hatsigmabound}, \eqref{Enorm*} and \eqref{ETHnorm} we derive
\begin{equation}
\label{hatxibound1}
(2\pi)^{\frac{3}{2}} \left| \widehat{\sigma}(-\xi) \right|
\leq \varepsilon \|E( ;0) \times \nu\|_{\TH(\partial\Omega)} (1+k) (\Vol(\Omega))^{1/2}
\leq \varepsilon (1+k)^{2} \Vol(\Omega), \quad
\textrm{if\ } |\xi| \leq 2k,
\end{equation}
and
\begin{equation}
\label{hatxibound2}
\begin{aligned}
(2\pi)^{\frac{3}{2}} \left| \widehat{\sigma}(-\xi) \right|
&\leq \varepsilon \|E( ;0) \times \nu\|_{\TH(\partial\Omega)} (1+k) \Big( \int_{\Omega} e^{\Xi\, e_{2} \cdot x} \,\mathrm{d}x \Big)^{1/2} \\
&\leq \varepsilon (1+k)^{2} \Big( \int_{\Omega} e^{-\Xi\, e_{2} \cdot x} \,\mathrm{d}x \Big)^{1/2} \Big( \int_{\Omega} e^{\Xi\, e_{2} \cdot x} \,\mathrm{d}x \Big)^{1/2} \\
&
\leq \varepsilon (1+k)^{2} \Vol_{2}(\Omega) \int_{-\frac{D}{2}}^{\frac{D}{2}} e^{-\Xi\, t} \,\mathrm{d}t,
& &\textrm{if\ } 2k < |\xi|.
\end{aligned}
\end{equation}

We first consider the case \textbf{a)}: $-\ln\varepsilon = \mathcal{E} < k$. By the Parseval identity
\begin{equation}
\label{sigmanorm_a}
\begin{aligned}
\|\sigma\|^{2}_{(0)}(\Omega)
&= \int_{|\xi| \leq 2k^{\alpha}} \left| \widehat{\sigma}(-\xi) \right|^{2} \,\mathrm{d}\xi
+ \int_{2k^{\alpha} < |\xi|} \left| \widehat{\sigma}(-\xi) \right|^{2} \,\mathrm{d}\xi \\
&\leq \frac{1}{(2\pi)^{3}} (\Vol(\Omega))^{2} (1+k)^{4} \frac{4\pi}{3} 8 k^{3\alpha} \varepsilon^{2}
+ \frac{M_{1}^{2}}{1+(2k^{\alpha})^{2}} \\
&\leq \frac{4}{3\pi^{2}} (\Vol(\Omega))^{2} (1+k)^{4} k^{3\alpha} \varepsilon^{2}
+ \frac{M_{1}^{2}}{1+2\mathcal{E}^{2\alpha}+2k^{2\alpha}}
\end{aligned}
\end{equation}
due to \eqref{hatxibound1} and the assumption that $\mathcal{E} < k$.

Now we handle the case \textbf{b)}: $k \leq \mathcal{E}$. Letting $\rho^{2} = \frac{\mathcal{E}^{2}}{D^{2}} + 4k^{2}$, since the function $\frac{e^{y}-e^{-y}}{y}$ is increasing when $y > 0$, we have
\begin{equation}
\label{Xiestimate}
\int_{-\frac{D}{2}}^{\frac{D}{2}} e^{-\Xi\, t} \,\mathrm{d}t
= \frac{e^{\frac{1}{2}D \Xi}-e^{-\frac{1}{2}D \Xi}}{\Xi} \leq D \frac{e^{\frac{\mathcal{E}}{2}}-e^{-\frac{\mathcal{E}}{2}}}{\mathcal{E}}
\end{equation}
when $2k < |\xi| \leq \rho$ (and hence $\Xi \leq \frac{\mathcal{E}}{D}$). As above, using \eqref{ETHnorm}, \eqref{hatxibound2}, \eqref{Xiestimate} we derive
\begin{equation*}
\begin{aligned}
\|\sigma\|^{2}_{(0)}(\Omega)
&= \int_{|\xi| \leq 2k} \left| \widehat{\sigma}(-\xi) \right|^{2} \,\mathrm{d}\xi
+ \int_{2k < |\xi| \leq \rho} \left| \widehat{\sigma}(-\xi) \right|^{2} \,\mathrm{d}\xi
+ \int_{\rho < |\xi|} \left| \widehat{\sigma}(-\xi) \right|^{2} \,\mathrm{d}\xi \\
&\leq \frac{4}{3\pi^{2}} (\Vol(\Omega))^{2} (1+k)^{4} k^{3} \varepsilon^{2}
+ \frac{1}{8\pi^{3}} (\Vol_{2}(\Omega))^{2} (1+k)^{4} \frac{D^{2}}{\mathcal{E}^{2}} \big(e^{\frac{\mathcal{E}}{2}}-e^{-\frac{\mathcal{E}}{2}}\big)^{2} \varepsilon^{2} \int_{2k < |\xi| \leq \rho} \,\mathrm{d}\xi
+ \frac{M_{1}^{2}}{1+\rho^{2}}.
\end{aligned}
\end{equation*}
We have
\begin{equation*}
\begin{aligned}
\int_{2k < |\xi| \leq \rho } \,\mathrm{d}\xi
&= \frac{4\pi}{3} \big( \rho^{3} - (2k)^{3} \big)
= \frac{4\pi}{3} \Big( \big(\frac{\mathcal{E}^{2}}{D^{2}}+4k^{2}\big)^{3/2} - (2k)^3 \Big) \\
&= \frac{4\pi}{3} \frac{\mathcal{E}^{3}}{D^{3}} \Big( \big(1+(2D\frac{k}{\mathcal{E}})^{2}\big)^{3/2} - (2D\frac{k}{\mathcal{E}})^{3} \Big)
~\leq~ 2\pi \frac{\mathcal{E}^{3}}{D^{3}} \big(1+(2D)^{2}\big)^{1/2},
\end{aligned}
\end{equation*}
due to the mean value theorem applied to the function $t^{\frac{3}{2}}$ and the assumption that $k \leq \mathcal{E}$. Indeed, we can use that by the mean value theorem (with some $0 < t^{*} < t$)
\begin{equation*}
(1+t)^{\frac{3}{2}} - t^{\frac{3}{2}}
= \frac{3}{2} (1+t^{*})^{\frac{1}{2}}
\leq \frac{3}{2} (1+t)^{\frac{1}{2}}
\end{equation*}
and let $t = (2D\frac{k}{\mathcal{E}})^{2}$. Summing up
\begin{equation}
\label{sigmanorm_b}
\|\sigma\|^{2}_{(0)}(\Omega)
\leq \frac{4}{3\pi^{2}} (\Vol(\Omega))^{2} (1+\mathcal{E})^{4} \mathcal{E}^{3} \varepsilon^{2}
+ \frac{1}{4\pi^{2}} (\Vol_{2}(\Omega))^{2} \frac{1}{D} \big(1+(2D)^{2}\big)^{1/2} (1+\mathcal{E})^{4} \mathcal{E} \varepsilon
+ \frac{M_{1}^{2}}{1+\frac{\mathcal{E}^{2}}{D^{2}}+4k^{2}}
\end{equation}
where we used that $\big(e^{\frac{\mathcal{E}}{2}}-e^{-\frac{\mathcal{E}}{2}}\big)^{2} \varepsilon^{2} = \big(e^{\mathcal{E}}+e^{-\mathcal{E}}-2\big) \varepsilon^{2} = \big(\varepsilon^{-1}+\varepsilon-2\big) \varepsilon^{2} \leq \varepsilon$, because $0 < \varepsilon < 1$.

Now the bound \eqref{bound1} follows from \eqref{sigmanorm_a} and \eqref{sigmanorm_b}.
\end{proof}


\subsection{A special TE mode}\label{se_main_TE}

Due to difficulties with a numerical solution of the full three-dimensional Maxwell system, we further consider a special TE mode where $\Omega$ is a cylindrical domain and the conductivity depends only on the transversal variables.

Let $\Omega = \Omega_{2} \times \mathbb{R}$ where $\Omega_{2}$ is a bounded $C^{2}$ domain in $\mathbb{R}^{2}$ and $\sigma(x) = \sigma(x',0)$, $x' = (x_{1},x_{2})$. If one seeks for a solution $(E,H)$ such that $E(x) = (0,0,E_{3}(x',0))$, $H(x)= (H_{1}(x',0),H_{2}(x',0),0)$, then the Maxwell system \eqref{system1} is equivalent to
\begin{equation}
\label{system1_eq}
{-\partial_{1} E_{3}} = \i \omega \mu_{0} H_{2}, \quad
\partial_{2} E_{3} = \i \omega \mu_{0} H_{1}, \quad
\partial_{1} H_{2} - \partial_{2} H_{1} = (-\i \omega \epsilon_{0} + \sigma) E_{3} \quad
\textrm{in\ } \Omega_{2},
\end{equation}
or
\begin{equation}
\label{pbI}
{-\Delta} E_{3} = (\omega^{2} \epsilon_{0} \mu_{0} + \i \omega \mu_{0} \sigma) E_{3}, \quad H_{1} = \frac{1}{\i\omega\mu_{0}} \partial_{2} E_{3}, \quad
H_{2} = \frac{-1}{\i\omega\mu_{0}} \partial_{1} E_{3} \quad
\textrm{in\ } \Omega_{2}
\end{equation}
where $\Delta = \partial_{1}^{2}+\partial_{2}^{2}$.

Let $\nu = (\nu',0)$ and $\nu' = (\nu_{1},\nu_{2})$ be the outward unit normal vector of $\partial\Omega_{2}$, since $E \times \nu = (-\nu_{2}E_{3},\nu_{1}E_{3},0)$, $H \times \nu = (0,0,\nu_{2}H_{1}-\nu_{1}H_{2}) = \frac{1}{\i\omega\mu_{0}}(0,0,\partial_{\nu'}E_{3})$, the boundary mapping \eqref{boundarymap} can be reformulated as
\begin{equation*}
\Lambda_{\textrm{TE}}(g) = \frac{1}{\i\omega\mu_{0}} \partial_{\nu'} E_{3} \quad \textrm{on\ } \partial\Omega_{2}
\end{equation*}
where $E_{3}$ is the solution of the Dirichlet boundary value problem
\begin{equation*}
{-\Delta} E_{3} = (\omega^{2} \epsilon_{0} \mu_{0} + \i \omega \mu_{0} \sigma) E_{3} \quad
\textrm{in\ } \Omega_{2}, \qquad
E_{3} = g \quad \textrm{on\ } \partial\Omega_{2}.
\end{equation*}

Similarly to the previous subsection, we again take an asymptotical expansion
\begin{equation*}
E_{3} = E_{3}( ;0) + E_{3}( ;1) + E_{3}( ;2)
\end{equation*}
where $E_{3}( ;0)$ satisfies the unperturbed problem
\begin{equation}
\label{E30}
{-\Delta} E_{3}( ;0) - \omega^{2} \epsilon_{0} \mu_{0} E_{3}( ;0) = 0 \quad
\textrm{in\ } \Omega_{2}, \qquad
E_{3}( ;0) = g \quad
\textrm{on\ } \partial\Omega_{2},
\end{equation}
and $E_{3}( ;1)$ solves the Dirichlet problem
\begin{equation}
\label{E31}
{-\Delta} E_{3}( ;1) - \omega^{2} \epsilon_{0} \mu_{0} E_{3}( ;1) = \i \omega \mu_{0} \sigma E_{3}( ;0) \quad
\textrm{in\ } \Omega_{2}, \qquad
E_{3}( ;1) = 0 \quad
\textrm{on\ } \partial\Omega_{2}.
\end{equation}
The remaining term is denoted by $E_{3}( ;2)$ which is comparably small under a high frequency $\omega$.

The linearised boundary mapping in the TE mode is defined by
\begin{equation*}
\Lambda'_{\textrm{TE}} \, g = \frac{1}{\i\omega\mu_{0}} \partial_{\nu'} E_{3}( ;1) \quad
\textrm{on\ } \partial\Omega_{2}.
\end{equation*}
Similar to the above subsection, we again introduce an "adjoint" solution to the unperturbed equation
\begin{equation}
\label{E30*}
{-\Delta} E^{*}_{3}( ;0) - \omega^{2} \epsilon_{0} \mu_{0} E^{*}_{3}( ;0) = 0 \quad
\textrm{in\ } \Omega_{2}.
\end{equation}
From the Green's formula and \eqref{E31}, \eqref{E30*} we have
\begin{equation}
\label{hatsigma2}
- \int_{\Omega_{2}} \sigma E_{3}( ;0) E^{*}_{3}( ;0) \,\mathrm{d}x'
= \frac{1}{\i\omega\mu_{0}} \int_{\partial\Omega_{2}} \partial_{\nu'} E_{3}( ;1) E^{*}_{3}( ;0) \,\mathrm{d}S'
= \int_{\partial\Omega_{2}} \big( \Lambda'_{\textrm{TE}} E_{3}( ;0) \big) E^{*}_{3}( ;0) \,\mathrm{d}S'.
\end{equation}

We also need the trace theorem for Sobolev spaces, such that
\begin{equation}
\label{trace}
\|E_{3}\|_{(\frac{1}{2})}(\partial\Omega_{2})
\leq C_{2} \|E_{3}\|_{(1)}(\Omega_{2})
\end{equation}
with a constant $C_{2}$ depending on the domain $\Omega_{2}$ and on the choice of the norm in $H^{\frac{1}{2}}(\partial\Omega_{2})$.

Let $\varepsilon_{1}$ be the (operator) norm of $\Lambda'_{\textrm{TE}}$ (from $H^{\frac{1}{2}}(\partial\Omega_{2})$ into $H^{-\frac{1}{2}}(\partial\Omega_{2})$) and noting that wave number $k = \omega\sqrt{\epsilon_{0}\mu_{0}}$. The increasing stability estimate for this TE mode is presented below.

\begin{theorem}
\label{thm2}
Let $\|\sigma\|_{(1)}(\Omega_2) \leq M_{1}$, $\sigma = 0$ on $\partial\Omega_{2}$, and the wave number $k \geq 1$. Furthermore, let $0 < \alpha \leq 1$ and $\varepsilon_{1} < 1$, then there holds
\begin{equation}
\label{bound2C}
\begin{aligned}
\|\sigma\|^{2}_{(0)}(\Omega_{2})
&\leq \frac{1}{\pi} C_{2}^{4} (\Vol_{2}(\Omega_{2}))^{2} \Big( (1-\chi_{k}(\mathcal{E}_{1})) (1+k)^{4} k^{2\alpha}
+ \chi_{k}(\mathcal{E}_{1}) (1+\mathcal{E}_{1})^{4} \mathcal{E}_{1}^{2} \Big) \varepsilon_{1}^{2} \\
&\qquad + \chi_{k}(\mathcal{E}_{1}) \frac{1}{4\pi} C_{2}^{4} D^{2} (1+\mathcal{E}_{1})^{4} \varepsilon_{1} \\
&\qquad + M_{1}^{2} \Big( (1-\chi_{k}(\mathcal{E}_{1})) \frac{1}{1+2\mathcal{E}_{1}^{2\alpha}+2k^{2\alpha}}
+ \chi_{k}(\mathcal{E}_{1}) \frac{1}{1+\frac{\mathcal{E}_{1}^{2}}{D^{2}}+4k^{2}} \Big).
\end{aligned}
\end{equation}
Here $\mathcal{E}_{1} = -\ln\varepsilon_{1}$ and $\chi_{k}(\mathcal{E}_{1}) = 0$ if $\mathcal{E}_{1} < k$ and $\chi_{k}(\mathcal{E}_{1}) = 1$ if $k \leq \mathcal{E}_{1}$.
\end{theorem}

\begin{proof}
The proof is similar to that of Theorem \ref{thm1} but we need to construct another type of CE solutions. We use the complex exponential solutions, which is similar to those in \cite{ILX2018}, below
\begin{equation}
\label{complexsol2}
E_{3}(x' ;0) = e^{\i \zeta' \cdot x'}, \quad
E^{*}_{3}(x' ;0) = -e^{\i \zeta'^{*} \cdot x'},
\end{equation}
then \eqref{complexsol2} solves \eqref{E30}, \eqref{E30*} if and only if
\begin{equation}
\label{ab2}
\zeta' \cdot \zeta' = \zeta'^{*} \cdot \zeta'^{*} = k^{2} = \omega^{2} \epsilon_{0} \mu_{0}.
\end{equation}
Let $\xi' \in \mathbb{R}^{2}$, and $\{ e_{1} = \frac{\xi'}{|\xi'|}, e_{2} \}$ be an orthonormal basis in $\mathbb{R}^{2}$ and
\begin{equation}
\label{ab2_choice}
\zeta' = \frac{|\xi'|}{2} e_{1} + \sqrt{k^{2}-\frac{|\xi'|^{2}}{4}} e_{2}, \quad
\zeta'^{*} = \frac{|\xi'|}{2} e_{1} - \sqrt{k^{2}-\frac{|\xi'|^{2}}{4}} e_{2}.
\end{equation}
Then \eqref{ab2} is satisfied and hence \eqref{complexsol2} solve \eqref{E30}, \eqref{E30*}.

In the two dimensional case
\begin{equation*}
\widehat{\sigma}(\xi')
= (2\pi)^{-1} \int_{\mathbb{R}^{2}} \sigma(x') \, e^{-\i \xi' \cdot x'} \,\mathrm{d}x',
\end{equation*}
so from \eqref{hatsigma2}, \eqref{complexsol2}, \eqref{ab2_choice} we yield
\begin{equation*}
2\pi \,\widehat{\sigma}(-\xi')
= - \int_{\Omega_{2}} \sigma E_{3}( ;0) E^{*}_{3}( ;0) \,\mathrm{d}x'
= \int_{\partial\Omega_{2}} \big( \Lambda'_{\textrm{TE}} E_{3}( ;0) \big) E^{*}_{3}( ;0) \,\mathrm{d}S'.
\end{equation*}
Now using the definition of the operator norm $\varepsilon_{1}$ we get
\begin{equation}
\label{2hatsigmabound}
\begin{aligned}
2\pi \left| \widehat{\sigma}(-\xi') \right|
&\leq \varepsilon_{1} \|E_{3}( ;0)\|_{(\frac{1}{2})}(\partial\Omega_{2}) \, \|E^{*}_{3}( ;0)\|_{(\frac{1}{2})}(\partial\Omega_{2}) \\
&\leq \varepsilon_{1} C_{2}^{2} \|E_{3}( ;0)\|_{(1)}(\Omega_{2}) \, \|E^{*}_{3}( ;0)\|_{(1)}(\Omega_{2})
\end{aligned}
\end{equation}
because of \eqref{trace}.

Using \eqref{complexsol2}, \eqref{ab2} and \eqref{ab2_choice}, we further obtain
\begin{equation}
\label{2Enorm}
\begin{aligned}
&\|E_{3}( ;0)\|^{2}_{(1)}(\Omega_{2})
= \int_{\Omega_{2}} (1+|\zeta'|^{2}) \,\mathrm{d}x'
= (1+k^{2}) \Vol_{2}(\Omega_{2}),
& &\textrm{if\ } |\xi'| \leq 2k, \\
&\|E_{3}( ;0)\|^{2}_{(1)}(\Omega_{2})
= (1+k^{2}) \int_{\Omega_{2}} e^{-\Xi'\, e_{2} \cdot x'} \,\mathrm{d}x',
& &\textrm{if\ } 2k < |\xi'|,
\end{aligned}
\end{equation}
where $\Xi' = \sqrt{|\xi'|^{2}-4k^{2}}$. Similarly,
\begin{equation}
\label{2Enorm*}
\begin{aligned}
&\|E^{*}_{3}( ;0)\|^{2}_{(1)}(\Omega_{2})
= \int_{\Omega_{2}} (1+|\zeta'|^{2}) \,\mathrm{d}x'
= (1+k^{2}) \Vol_{2}(\Omega_{2}),
& &\textrm{if\ } |\xi'| \leq 2k, \\
&\|E^{*}_{3}( ;0)\|^{2}_{(1)}(\Omega_{2})
= (1+k^{2}) \int_{\Omega_{2}} e^{\Xi'\, e_{2} \cdot x'} \,\mathrm{d}x',
& &\textrm{if\ } 2k < |\xi'|.
\end{aligned}
\end{equation}

Now combining \eqref{2hatsigmabound}, \eqref{2Enorm} and \eqref{2Enorm*} we derive
\begin{equation}
\label{2hatxibound1}
2\pi \left| \widehat{\sigma}(-\xi') \right|
\leq \varepsilon_{1} C_{2}^{2} (1+k)^{2} \Vol_{2}(\Omega_{2}), \quad
\textrm{if\ } |\xi'| \leq 2k,
\end{equation}
and
\begin{equation}
\label{2hatxibound2}
2\pi \left| \widehat{\sigma}(-\xi') \right|
\leq \varepsilon_{1} C_{2}^{2} (1+k)^{2} D \int_{-\frac{D}{2}}^{\frac{D}{2}} e^{-\Xi'\, t'} \,\mathrm{d}t', \quad
\textrm{if\ } 2k < |\xi'|.
\end{equation}

We first consider the case \textbf{a)}: $\mathcal{E}_{1} < k$. We have
\begin{equation}
\label{2sigmanorma}
\begin{aligned}
\|\sigma\|^{2}_{(0)}(\Omega_{2})
&= \int_{|\xi'| \leq 2k^{\alpha}} \left| \widehat{\sigma}(-\xi') \right|^{2} \,\mathrm{d}\xi'
+ \int_{2k^{\alpha} < |\xi'|} \left| \widehat{\sigma}(-\xi') \right|^{2} \,\mathrm{d}\xi' \\
&\leq \frac{1}{(2\pi)^{2}} C_{2}^{4} (\Vol_{2}(\Omega_{2}))^{2} (1+k)^{4} \pi 4 k^{2\alpha} \varepsilon_{1}^{2}
+ \frac{M_{1}^{2}}{1+(2k^{\alpha})^{2}} \\
&\leq \frac{1}{\pi} C_{2}^{4} (\Vol_2(\Omega_{2}))^{2} (1+k)^{4} k^{2\alpha} \varepsilon_{1}^{2}
+ \frac{M_{1}^{2}}{1+2\mathcal{E}_{1}^{2\alpha}+2k^{2\alpha}}
\end{aligned}
\end{equation}
due to \eqref{2hatxibound1} and to the assumption that $\mathcal{E}_{1} < k$.

To handle the case \textbf{b)}: $k \leq \mathcal{E}_{1}$, we let $\rho_{1}^{2} = \frac{\mathcal{E}_{1}^{2}}{D^{2}} + 4k^{2}$. Since the function $\frac{e^{y}-e^{-y}}{y}$ is increasing when $y > 0$, we have
\begin{equation}
\label{2Xiestimate}
\int_{-\frac{D}{2}}^{\frac{D}{2}} e^{-\Xi'\, t'} \,\mathrm{d}t'
= \frac{e^{\frac{1}{2}D \Xi'}-e^{-\frac{1}{2}D \Xi'}}{\Xi'}
\leq D \frac{e^{\frac{\mathcal{E}_{1}}{2}}-e^{-\frac{\mathcal{E}_{1}}{2}}}{\mathcal{E}_{1}}
\end{equation}
when $2k < |\xi'| \leq \rho_{1}$ (and then $\Xi' \leq \frac{\mathcal{E}_{1}}{D}$). Hence as above when deriving \eqref{2sigmanorma}, using \eqref{2hatxibound2}, \eqref{2Xiestimate}, we obtain
\begin{equation*}
\begin{aligned}
\|\sigma\|^{2}_{(0)}(\Omega_{2})
&= \int_{|\xi'| \leq 2k} \left| \widehat{\sigma}(-\xi') \right|^{2} \,\mathrm{d}\xi'
+ \int_{2k < |\xi'| \leq \rho_{1}} \left| \widehat{\sigma}(-\xi') \right|^{2} \,\mathrm{d}\xi'
+ \int_{\rho_{1} < |\xi'|} \left| \widehat{\sigma}(-\xi') \right|^{2} \,\mathrm{d}\xi' \\
&\leq \frac{1}{\pi} C_{2}^{4} (\Vol_{2}(\Omega_{2}))^{2} (1+k)^{4} k^{2} \varepsilon_{1}^{2} \\
&\qquad + \frac{1}{4\pi^{2}} C_{2}^{4} D^{2} (1+k)^{4} \frac{D^{2}}{\mathcal{E}_{1}^{2}} \big(e^{\frac{\mathcal{E}_{1}}{2}}-e^{-\frac{\mathcal{E}_{1}}{2}}\big)^{2} \varepsilon_{1}^{2} \int_{2k < |\xi'| \leq \rho_{1}} \,\mathrm{d}\xi'
+ \frac{M_{1}^{2}}{1+\rho_{1}^{2}}.
\end{aligned}
\end{equation*}
Since $\rho_{1}^{2} = \frac{\mathcal{E}_{1}^{2}}{D^{2}} + 4k^{2}$, there holds
\begin{equation*}
\int_{2k < |\xi'| \leq \rho_{1}} \,\mathrm{d}\xi'
= \pi \big( \rho_{1}^{2} - (2k)^{2} \big)
= \pi \frac{\mathcal{E}_{1}^{2}}{D^{2}}.
\end{equation*}
Simple calculation then shows that when $k \leq \mathcal{E}_{1}$, there holds
\begin{equation}
\label{2sigmanormb}
\|\sigma\|^{2}_{(0)}(\Omega_{2})
\leq \frac{1}{\pi} C_{2}^{4} (\Vol_{2}(\Omega_{2}))^{2} (1+\mathcal{E}_{1})^{4} \mathcal{E}_{1}^{2} \varepsilon_{1}^{2}
+ \frac{1}{4\pi} C_{2}^{4} D^{2} (1+\mathcal{E}_{1})^{4} \varepsilon_{1}
+ \frac{M_{1}^{2}}{1+\frac{\mathcal{E}_{1}^{2}}{D^{2}}+4k^{2}}
\end{equation}
because $\varepsilon_{1} < 1$. The bounds \eqref{2sigmanorma}, \eqref{2sigmanormb} imply \eqref{bound2}.
\end{proof}

Since the trace operator is continuous from $H^{1}(\Omega_{2})$ into $H^{\frac{1}{2}}(\partial\Omega_{2})$, one of natural choices  of a norm is
\begin{equation*}
\|E_3\|_{(\frac{1}{2})}(\partial\Omega_{2}) = \inf \|\tilde{E}_3\|_{(1)}(\Omega_{2})
\end{equation*}
over all $\tilde{E}_3 \in H^{1}(\Omega_{2})$ with $\tilde{E}_3 = E_3$ on $\partial\Omega_{2}$. Then $C_{2} \leq 1$ in \eqref{trace} and Theorem \ref{thm2} implies

\begin{corollary}
\label{corollary}
Let $\|\sigma\|_{(1)}(\Omega_2) \leq M_{1}$, $\sigma = 0$ on $\partial\Omega_{2}$, and the wave number $k \geq 1$. Let $0 < \alpha \leq 1$ and $\varepsilon_{1} < 1$, then the following estimate holds true
\begin{equation}
\label{bound2}
\begin{aligned}
\|\sigma\|^{2}_{(0)}(\Omega_{2})
&\leq \frac{1}{\pi} (\Vol_{2}(\Omega_{2}))^{2} \Big( (1-\chi_{k}(\mathcal{E}_{1})) (1+k)^{4} k^{2\alpha}
+ \chi_{k}(\mathcal{E}_{1}) (1+\mathcal{E}_{1})^{4} \mathcal{E}_{1}^{2} \Big) \varepsilon_{1}^{2} \\
&\qquad + \chi_{k}(\mathcal{E}_{1}) \frac{1}{4\pi} D^{2} (1+\mathcal{E}_{1})^{4} \varepsilon_{1} \\
&\qquad + M_{1}^{2} \Big( (1-\chi_{k}(\mathcal{E}_{1})) \frac{1}{1+2\mathcal{E}_{1}^{2\alpha}+2k^{2\alpha}}
+ \chi_{k}(\mathcal{E}_{1}) \frac{1}{1+\frac{\mathcal{E}_{1}^{2}}{D^{2}}+4k^{2}} \Big).
\end{aligned}
\end{equation}
Here $\mathcal{E}_{1} = -\ln\varepsilon_{1}$ and $\chi_{k}(\mathcal{E}_{1}) = 0$ if $\mathcal{E}_{1} < k$ and $\chi_{k}(\mathcal{E}_{1}) = 1$ if $k \leq \mathcal{E}_{1}$.
\end{corollary}

We comment on increasing stability estimates in Theorems \ref{thm1} and \ref{thm2}. As one can observe, these estimates contain three terms. The first one is either a Lipschitz term whose constant grows polynomially with respect to the wave number $k$ or a sub-Lipschitz term associated with a logarithmic factor $\mathcal{E} = -\ln\varepsilon$. The second item is a H\"{o}lder term over $\varepsilon$ associated with another logarithmic factor $\mathcal{E}$. The last item is a negative logarithmic term $O(1/\mathcal{E}^{2\alpha})$, $0< \alpha \leq 1$ which commonly arises in the stability estimate of elliptic inverse boundary value problems. Nevertheless, this last term also has another order of $O(1/(\mathcal{E}^{2\alpha}+k^{2\alpha}))$, $0< \alpha \leq 1$ which may provide a better error bound if the wave number $k$ is large. Similar increasing stability estimate appears also in the linearised Schr\"odinger potential problem for the acoustic wave equation as shown in \cite{ILX2018}.


\section{Numerical aspects for the linearised inverse problem}\label{se_numer}

Our main goal in this section is to propose a reconstruction algorithm and present some numerical evidences confirming the derived increasing stability in Section \ref{se_main}.

\subsection{Numerical reconstruction algorithm}

The Fourier based reconstruction algorithm, to be proposed below, relies on the equality \eqref{hatsigma} where the linearised boundary mapping $\Lambda'$ provides the boundary data
\begin{align*}
\Lambda' : E( ;0) \times \nu |_{\partial\Omega} \mapsto H( ;1) \times \nu |_{\partial\Omega}.
\end{align*}
Noticing that in the first-order subproblem \eqref{system3}, $H( ;1)$ depends on the unknown conductivity $\sigma(x)$, we shall approximate its boundary value in the form of
\begin{align*}
H( ;1) \times \nu |_{\partial\Omega} \approx \left( H \times \nu - H( ;0) \times \nu \right)|_{\partial\Omega}.
\end{align*}
Such a linearised approximation has been widely implemented in inverse conductivity problems, for instance in \cite{DS1994} and very recently in \cite{ILX2018}. After this preparation, we could present the reconstruction algorithm for the linearised Maxwell system and aim at recovering Fourier modes of the conductivity $\sigma(x)$. This algorithm firstly chooses several discrete sets of lengths and angles of the vectors in the phase space. For instance, we choose a discrete and finite length set
\begin{align*}
\{\kappa_{\ell}\}_{\ell=1}^{M} \subset (0, \mathcal{K} \,]
\quad \textrm{for any fixed wave number \ } k = \omega \sqrt{\mu_{0} \epsilon_{0}}.
\end{align*}
Here $\mathcal{K}$ is the maximum length of the modulus $|\xi|$, whose choice will be specified later. Choosing three unit vector (or angle) sets
\begin{align*}
\{e^{1}_{s}\}_{s=1}^{N} \subset \mathbb{S}^{n-1}, \quad
\{e^{2}_{s}\}_{s=1}^{N} \subset \mathbb{S}^{n-1}
\quad \textrm{and} \quad \{e^{3}_{s}\}_{s=1}^{N} \subset \mathbb{S}^{n-1},
\end{align*}
satisfying $e^{1}_{s} \cdot e^{2}_{s} = 0$ and $e^{3}_{s} = e^{1}_{s} \times e^{2}_s$, we denote
\begin{align*}
\xi^{\langle \ell;s \rangle} = \kappa_{\ell} \, e^{1}_{s},
\quad
\zeta^{\langle \ell;s \rangle} = \frac{\kappa_{\ell}}{2} e^{1}_{s} + \sqrt{k^{2} - \frac{\kappa_{\ell}^{2}}{4}} \, e^{2}_{s},
\quad
\zeta^{\langle \ell;s \rangle}_{*} = \frac{\kappa_{\ell}}{2} e^{1}_{s} - \sqrt{k^{2} - \frac{\kappa_{\ell}^{2}}{4}} \, e^{2}_{s},
\end{align*}
and
\begin{align*}
&a^{\langle s \rangle} = +e^{3}_{s}, \quad
b^{\langle \ell;s \rangle} = \frac{-1}{\omega \mu_0} \left( \frac{\kappa_{\ell}}{2} e^{2}_{s} - \sqrt{k^2-\frac{\kappa_{\ell}^{2}}{4}} \, e^{1}_{s} \right), \\
&a^{\langle s \rangle}_{*} = -e^{3}_{s}, \quad
b^{\langle \ell;s \rangle}_{*} = \frac{+1}{\omega \mu_0} \left( \frac{\kappa_{\ell}}{2} e^{2}_{s} + \sqrt{k^2-\frac{\kappa_{\ell}^{2}}{4}} \, e^{1}_{s} \right),
\end{align*}
which are vectors (or points) chosen in the phase space, while $\ell = 1,2,\cdots,M$ and $s = 1,2,\cdots,N$. More precisely, the superscript notation $\cdot^{\langle \ell;s \rangle}$ will be referred to a vector $\xi^{\langle \ell;s \rangle}$ with the $\ell$th length $\kappa_{\ell}$ and the $s$th angle $e^{1}_{s}$. To realize the inverse Fourier transform, we choose a numerical quadrature rule by a suitable choice of the weights $w^{\langle \ell;s \rangle}$ according to these points $\xi^{\langle \ell;s \rangle}$.

We summarize our reconstruction algorithm below.
\vspace{10pt}
\hrule\hrule
\vspace{8pt}
{\parindent 0pt \bf Algorithm 1: Reconstruction Algorithm for the Linearised Inverse Conductivity Problem}
\vspace{8pt}
\hrule
\vspace{8pt}
{\parindent 0pt \bf Input:} %
$\{\kappa_{\ell}\}_{\ell=1}^{M}$, %
$\{e^{1}_{s}\}_{s=1}^{N}$, $\{e^{2}_{s}\}_{s=1}^{N}$, $\{e^{3}_{s}\}_{s=1}^{N}$ and %
weights $w^{\langle \ell;s \rangle}$; \\[5pt]%
{\parindent 0pt \bf Output:} %
Reconstructed conductivity $\sigma_{\rm inv}= \sigma^{\langle M+1;1 \rangle}$. \\[-5pt]%
\begin{enumerate}
  \item[1:] $\,$ Set $\sigma^{\langle 1;1 \rangle} := 0$; %
  \item[2:] $\,$ {\bf For} $\ell = 1,2,\cdots,M$ (length~updating) %
  \item[3:] $\,$ \quad {\bf For} $s = 1,2,\cdots,N$ (angle~updating) %
  \item[4:] $\,$ \quad \quad Choose $E( ;0) := a^{\langle s \rangle} \exp \{ \i \zeta^{\langle \ell;s \rangle} \cdot x \}$, $H( ;0) := b^{\langle \ell;s \rangle} \exp \{ \i \zeta^{\langle \ell;s \rangle} \cdot x \}$; %
  \item[5:] $\,$ \quad \quad Measure boundary data $H \times \nu \big|_{\partial\Omega}$ of the Maxwell system \eqref{system1} given boundary data $E(;0) \times \nu \big|_{\partial\Omega}$;
  \item[6:] $\,$ \quad \quad Calculate approximated linearised boundary data $H \times \nu - H(;0) \times \nu$ on the boundary $\partial\Omega$;
  \item[7:] $\,$ \quad \quad Choose $E^{*}( ;0) := a^{\langle s \rangle}_{*} \exp \{ \i \zeta^{\langle \ell;s \rangle}_{*} \cdot x \}$, $H^{*}( ;0) := b^{\langle \ell;s \rangle}_{*} \exp \{ \i \zeta^{\langle \ell;s \rangle}_{*} \cdot x \}$; %
  \item[8:] $\,$ \quad \quad Compute $\widehat{\sigma}(-\xi^{\langle \ell;s \rangle}) \approx -(a^{\langle s \rangle} \cdot a^{\langle s \rangle}_{*})^{-1} \int_{\partial \Omega} ( H \times \nu - H( ;0) \times \nu ) \cdot E^{*}( ;0) \,\mathrm{d}S$; %
  \item[9:] $\,$ \quad \quad Update $\sigma^{\langle \ell;s+1 \rangle} = \sigma^{\langle \ell;s \rangle} + w^{\langle \ell;s \rangle} \widehat{\sigma}(-\xi^{\langle \ell;s \rangle}) \, \exp \{ \i \xi^{\langle \ell;s \rangle} \cdot x \}$; %
  \item[10:] $\,$ \quad {\bf End} %
  \item[11:] $\,$ \quad Set $\sigma^{\langle \ell+1;1 \rangle} := \sigma^{\langle \ell;N+1 \rangle}$; %
  \item[12:] $\,$ {\bf End}. %
\end{enumerate}
\vspace{8pt}
\hrule\hrule
\vspace{10pt}

As one can observe, in Algorithm 1, the truncation threshold value $\mathcal{K}$ determines the highest Fourier modes of the reconstructed conductivity. Nevertheless, one can not choose $\mathcal{K}$ arbitrary large since CE solutions would become highly oscillating. Similarly to the investigation of acoustic wave equations in \cite{ILX2018} the choice of $\mathcal{K} = 2k$ is more appropriate to obtain a stable reconstruction and we provide the numerical evidence later.


\subsection{Numerics of the forward problems and linearised boundary data}

To simplify the numerical calculation, we consider the TE mode in Subsection \ref{se_main_TE} and focus on the electromagnetic fields in a cylindrical domain $\Omega$, when the conductivity $\sigma(x)$ depends only on the transversal variables. By choosing $\Omega = \Omega_{2} \times \mathbb{R}$ and $\sigma(x) = \sigma(x',0)$ with $x' = (x_{1},x_{2})$, we reduce the full three-dimensional Maxwell system \eqref{system1} to a simplified TE mode \eqref{system1_eq}, or more precisely, an uncoupled Helmholtz-type equation \eqref{pbI} for the electric field $E_{3}$ by eliminating the magnetic fields $H_{1}$ and $H_{2}$. In fact, the magnetic fields $H_{1}$ and $H_{2}$ can also be calculated by this simplified TE mode, i.e., the formula \eqref{pbI} which are gradient fields of the electric field $E_{3}$.

\begin{figure}[htb]
\centering
\Large
\begin{tikzpicture}[>=latex,scale=1.0,x={(1,0)},y={(0,1)},z={(-0.3,-0.4)}]
\draw[blue, ->, ultra thick, dashed] (0,0,0) -- (0,-2,0) node[pos=1.0,right] {$a^{*}~(E^{*}( ;0))$};
\fill[gray!20!white,rounded corners=20,fill opacity=0.7] (3.8,0,4.5) -- (-3.2,0,4.5) -- (-3.2,0,-3.7) -- (3.8,0,-3.7) -- cycle;
\draw[->] (0,0,0) -- (0,0,4) node[pos=1.0,left] {$x_{1}$};
\draw[->] (0,0,0) -- (2.5,0,0) node[pos=1.1] {$x_{2}$};
\draw[->] (0,0,0) -- (0,2.5,0) node[pos=1.1] {$x_{3}$};
\draw[->, ultra thick] (0,0,0) -- (3,0,3) node[pos=1.0,right] {$\xi'=|\xi'|e_{1}$};
\draw[->, ultra thick] (0,0,0) -- (3,0,-3) node[pos=1.0,right] {$\xi'_{\perp}=|\xi'|e_{2}$};
\draw[blue, ->, ultra thick] (0,0,0) -- (1,0,2) node[pos=1.0,below] {$\zeta'^{*}$};
\draw[red, ->, ultra thick] (0,0,0) -- (2,0,1) node[pos=1.0,right] {$\zeta'$};
\draw[thick, dashed] (1,0,2) -- (3,0,3) -- (2,0,1);
\draw[red, ->, ultra thick] (0,0,0) -- (0,2,0) node[pos=1.0,right] {$a~(E( ;0))$};
\draw[red, ->, ultra thick] (0,0,0) -- (-1,0,2) node[pos=1.0,left] {$b~(H( ;0))$};
\draw[blue, ->, ultra thick] (0,0,0) -- (2,0,-1) node[pos=1.0,right] {$b^{*}~(H^{*}( ;0))$};
\draw[fill=black] (0,0,0) circle (0.07) node at (-0.3,0.2,0) {\normalsize$O$};
\end{tikzpicture}
\caption{TE mode.}
\label{fig_TE_mode}
\end{figure}

To have a clear view of the numerical setting we provide Figure \ref{fig_TE_mode} illustrating the relation of CE solutions between the TE mode and the full Maxwell system in $\mathbb{R}^{3}$. Indeed, according to the CE solutions \eqref{complexsol} given (\ref{ab_choice}), the fields $E( ;0)$, $E^{*}( ;0)$, $H( ;0)$ and $H^{*}( ;0)$ represent the planar waves traveling in the direction of the wave vectors $\zeta'$ and $\zeta'^{*}$ respectively. The vectors $a$, $a^{*}$ are denoted as the directions of the electric fields $E( ;0)$, $E^{*}( ;0)$, and the vectors $b$, $b^{*}$ are denoted as the directions of the magnetic fields $H( ;0)$, $H^{*}( ;0)$, by the Faraday's Law. In the TE mode, except that the vectors $a$ and $a^{*}$ are along the $x_{3}$ direction, the other vectors lie in the $x_{1}$-$x_{2}$ plane. We emphasize that such a TE mode allows us to reduce the algorithmic complexity and substantially save the computational cost. For similar treatment and further applications, we refer to \cite{R2012}.

We first provide some numerics of the forward problems and validate the performance of the linearised boundary data. In particular, this subsection focuses on Steps 5-6 in Algorithm 1 and especially clarify impacts of the frequency towards the linearised boundary data.

The domain $\Omega_{2}$ is chosen as a disk centred at origin with a radius $0.7$ m, i.e. $\Omega_{2}: = B_{0.7}(0) \subset [{-0.7},0.7]^{2}$. The conductivity $\sigma$ is chosen as
\begin{align}
\label{conductivity}
\sigma(x')
= \big| 3(1-x_{1})^{2} \exp\{-x_{1}^{2}-(x_{2}+1)^{2}\}
- (2 x_{1} - 10 x_{1}^{3} - 10 x_{2}^{5}) \exp\{-x_{1}^{2}-x_{2}^{2}\}
- \frac{1}{3} \exp\{-(x_{1}+1)^{2}-x_{2}^{2}\} \big|.
\end{align}
As displayed in the left panel of Figure \ref{fig_sigma}, this conductivity is non-negative and has several Gaussian-type peaks. The boundary $\partial\Omega_{2}$ is also indicated by the red circle with a radius $0.7$ in the same figure.
Noticing that such a conductivity can be represented in the phase space by the variable $\xi' = (\xi_{1},\xi_{2})$, we thus plot $19$ inclined segments in the middle panel of Figure \ref{fig_sigma} where each star indicates a point $\xi'$ in the phase space satisfying $|\xi'| \leq 50$ with different angles. Near these points, the Fourier modes can be calculated explicitly whose absolute values are displayed in the right panel of Figure \ref{fig_sigma}. We emphasize that each curve there collects the absolute Fourier modes of points corresponding to an inclined segment in the middle panel of Figure \ref{fig_sigma}. These Fourier modes are of importance in current work and Algorithm 1 aims at recovering as many of them as possible, c.f. Step 8 ibid.

\begin{figure}[!htb]
\centering
\includegraphics[width=1.0\textwidth]{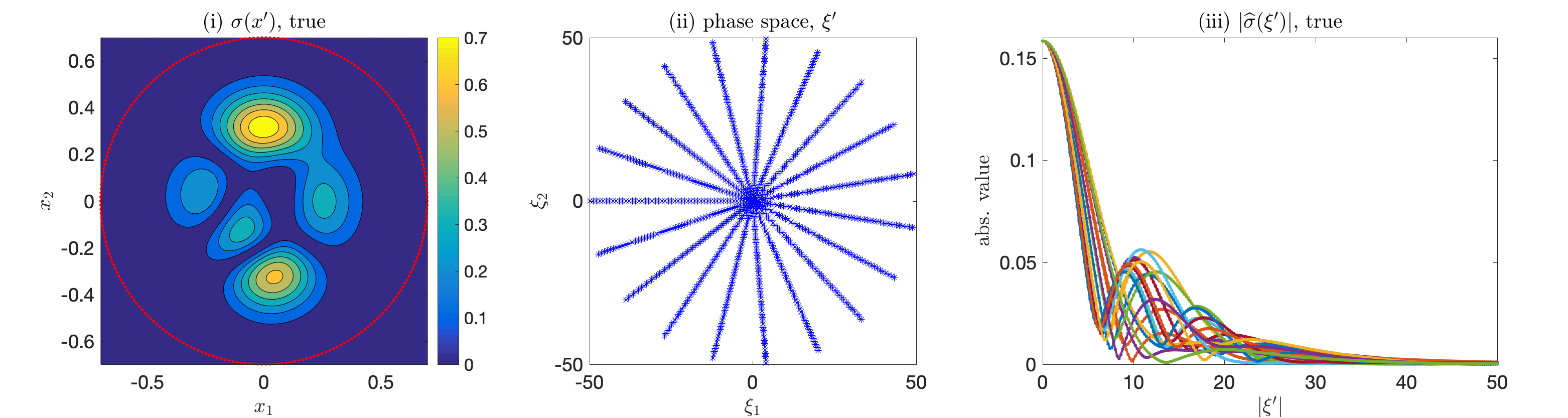}
\caption{Left: the true conductivity $\sigma(x')$. Middle: $19$ inclined segments in the phase space. Right: the associated absolute Fourier modes with respect to $|\xi'|$.}
\label{fig_sigma}
\end{figure}

Below we provide some numerical solutions $(E,H)$ with $E = (0,0,E_{3})$ and $H = (H_{1},H_{2},0)$ of the Maxwell system \eqref{system1_eq} given the conductivity \eqref{conductivity} and one particular incident plane wave with $\xi' = (-1,0)$.
To calculate the electric and magnetic fields induced by the conductivity we take constant electrical permittivity and magnetic permeability to be $\epsilon_{0} = \mu_{0} = 1$ and consequently the wave number $k = \omega$. Considering numerical discretization of the forward problem, we choose a fine grids with $200 \times 200$ equal-distance points for solving the boundary value problem \eqref{pbI} with Dirichlet boundary data in the domain $[{-0.7},0.7]^{2}$.
In fact, it is important to have accurate gradient in such cases, and a second order accurate embedded boundary method of finite difference schemes on an irregular domain, c.f. \cite{GFCK2002,KPY2002,KP2006}, are used for solving the forward problem \eqref{pbI} on the disk $\Omega_{2}$.
To highlight the influence of the frequency, we choose $\omega = 5, 15$ Hz respectively and present the real parts of electric and magnetic fields in Figures \ref{fig_Ew05}-\ref{fig_Ew15}. As one can observe, when the frequency becomes large, i.e. $\omega = 15$, high frequency patterns appear in both the electric and magnetic fields.

\begin{figure}[!htb]
\centering
\includegraphics[width=1.0\textwidth]{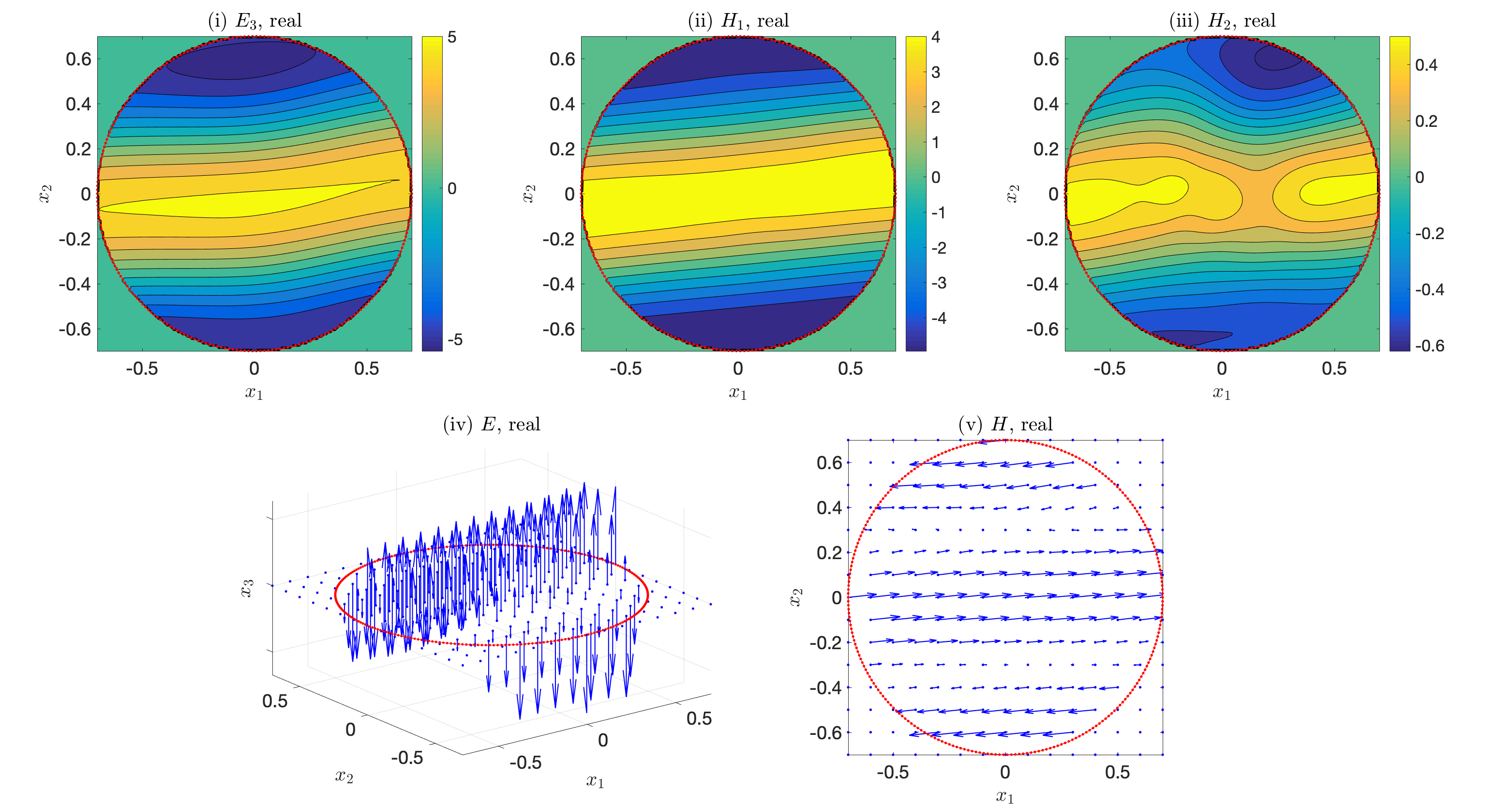}
\caption{[$\xi' = (-1,0)$, $\omega = 5$ Hz] Real parts of fields $E = (0,0,E_{3})$ and $H = (H_{1},H_{2},0)$ satisfying the Maxwell system \eqref{system1_eq} with the conductivity $\sigma$ in (\ref{conductivity}).
Upper row: real parts of $E_{3}$ (V/m), $H_{1}$ (A/m) and $H_{2}$ (A/m) (from left to right).
Bottom row: real parts of $E$ (V/m) and $H$ (A/m) (from left to right).}
\label{fig_Ew05}
\end{figure}

\begin{figure}[!htb]
\centering
\includegraphics[width=1.0\textwidth]{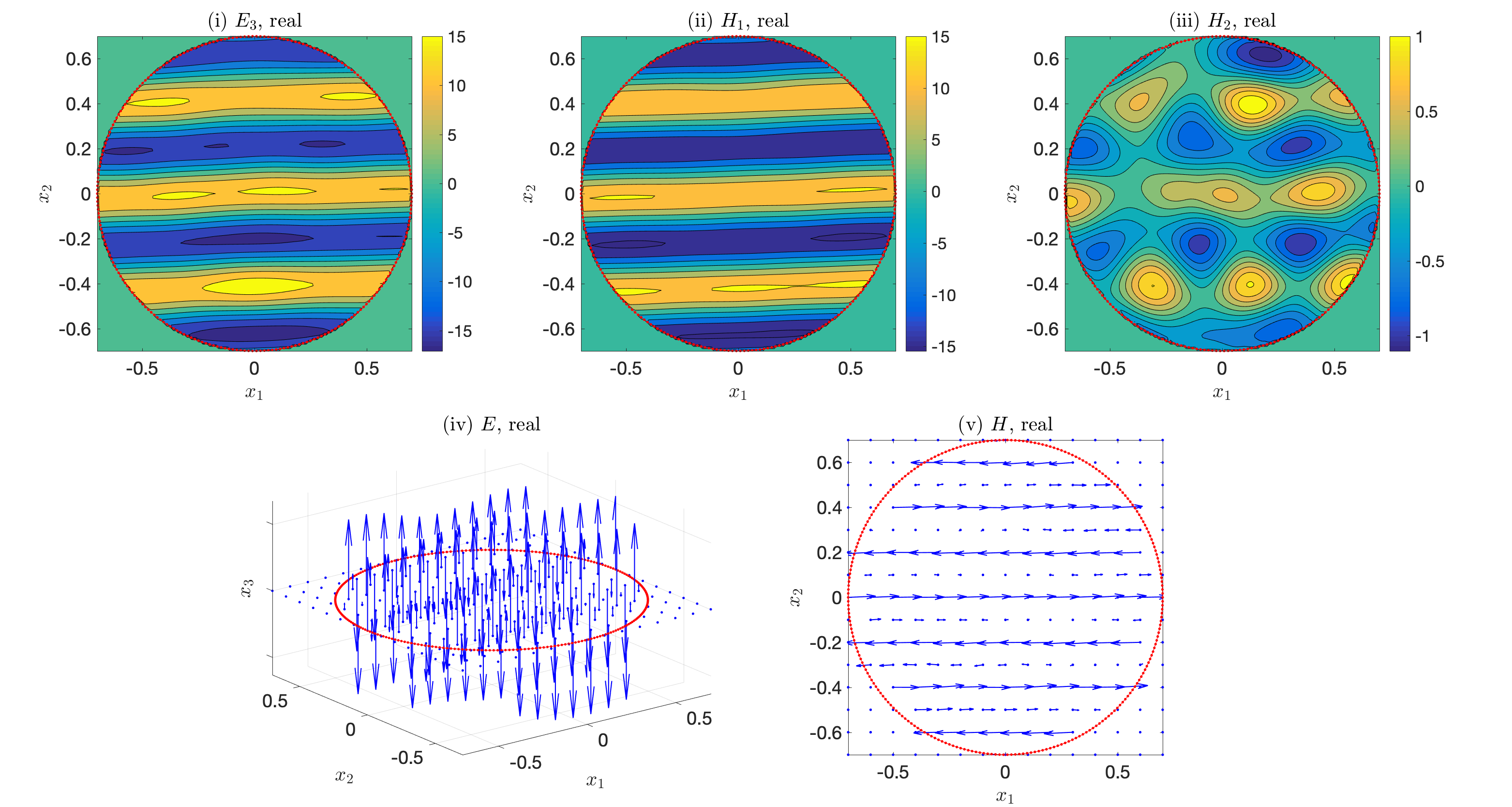}
\caption{[$\xi' = (-1,0)$, $\omega = 15$ Hz] Real parts of fields $E = (0,0,E_{3})$ and $H = (H_{1},H_{2},0)$ satisfying the Maxwell system  \eqref{system1_eq} with the conductivity $\sigma$ in (\ref{conductivity}).
Upper row: real parts of $E_{3}$ (V/m), $H_{1}$ (A/m) and $H_{2}$ (A/m) (from left to right).
Bottom row: real parts of $E$ (V/m) and $H$ (A/m) (from left to right).}
\label{fig_Ew15}
\end{figure}

To further illustrate the difference between the electric field $E_{3}$ and its unperturbed approximation $E_{3}( ;0)$, we present their difference $E_{3} - E_{3}( ;0)$ in Figures \ref{fig_dEHw05}-\ref{fig_dEHw15}, for two choices of the frequency $\omega = 5, 15$. Noticing that the magnetic fields $H_{1}$ and $H_{2}$ obey
\begin{align*}
H_{1} = \frac{1}{\i\omega\mu_{0}} \partial_{2} E_{3}, \quad
H_{2} = \frac{-1}{\i\omega\mu_{0}} \partial_{1} E_{3} \quad
\textrm{in\ } \Omega_{2},
\end{align*}
as well as their unperturbed approximation
\begin{align*}
H_{1}( ;0) = \frac{1}{\i\omega\mu_{0}} \partial_{2} E_{3}( ;0), \quad
H_{2}( ;0) = \frac{-1}{\i\omega\mu_{0}} \partial_{1} E_{3}( ;0) \quad
\textrm{in\ } \Omega_{2},
\end{align*}
we collect the real part of difference $H_{1} - H_{1}(;0)$ and $H_{2} - H_{2}(;0)$ in Figures \ref{fig_dEHw05}-\ref{fig_dEHw15} as well. When the frequency grows, i.e. $\omega = 15$, high frequency patterns essentially appear in the difference of electromagnetic fields which are also reflected on the boundary data. Here, we recall the approximation of $\partial_{\nu'} E_{3}( ;1)$ below
\begin{align*}
\partial_{\nu'} E_{3}( ;1) \approx \partial_{\nu'} E_{3} - \partial_{\nu'} E_{3}( ;0) \quad \textrm{on\ } \partial\Omega_{2},
\end{align*}
which realizes the linearised boundary mapping numerically.
In Figures \ref{fig_dEHw05}-\ref{fig_dEHw15}, we present the real and imaginary parts of $\partial_{\nu'} E_{3} - \partial_{\nu'} E_{3}( ;0)$ on the boundary $\partial\Omega_{2}$ where one can observe its tendency under growing frequencies.

\begin{figure}[!htb]
\centering
\includegraphics[width=1.0\textwidth]{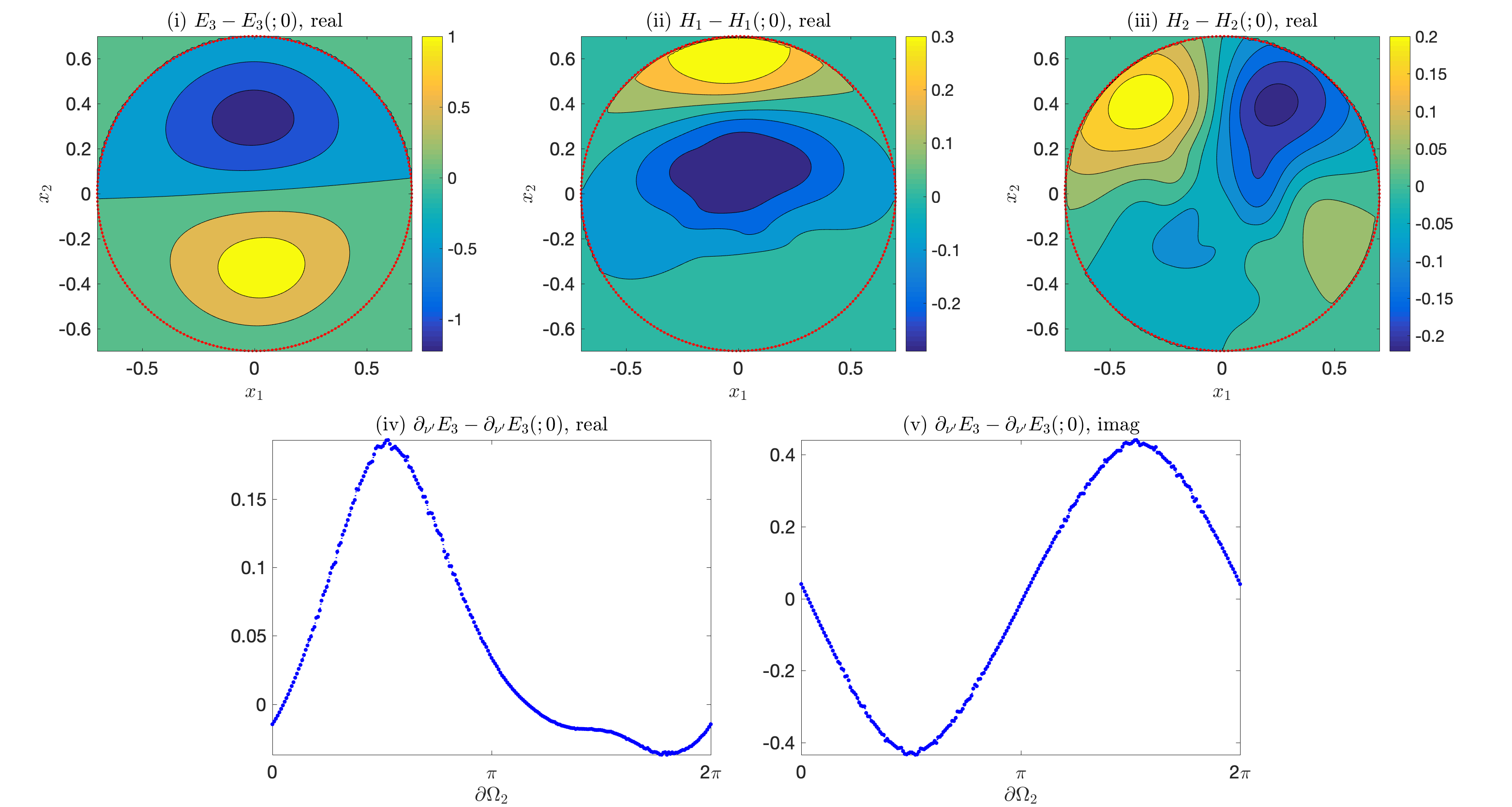}
\caption{[$\xi' = (-1,0)$, $\omega = 5$ Hz]
Upper row: real part of the difference of electromagnetic fields $E_{3} - E_{3}( ;0)$ (V/m), $H_{1} - H_{1}(;0)$ (A/m) and $H_{2} - H_{2}(;0)$ (A/m) (from left to right).
Bottom row: real part (left) and imaginary part (right) of the linearised Neumann boundary data $\partial_{\nu'} E_{3} - \partial_{\nu'} E_{3}( ;0)$ on $\partial\Omega_{2}$.}
\label{fig_dEHw05}
\end{figure}

\begin{figure}[!htb]
\centering
\includegraphics[width=1.0\textwidth]{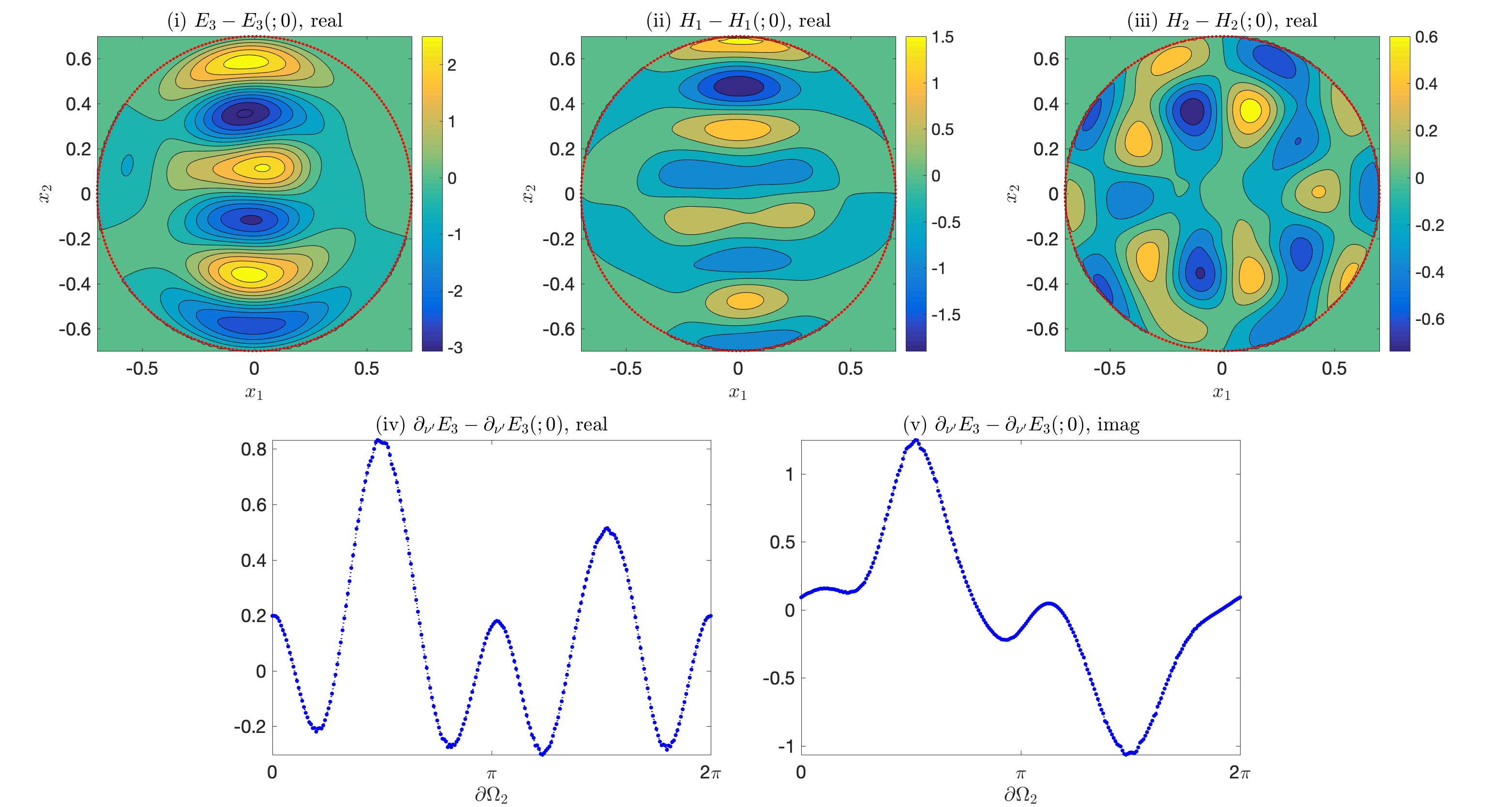}
\caption{[$\xi' = (-1,0)$, $\omega = 15$ Hz]
Upper row: real part of the difference of electromagnetic fields $E_{3} - E_{3}( ;0)$ (V/m), $H_{1} - H_{1}(;0)$ (A/m) and $H_{2} - H_{2}(;0)$ (A/m) (from left to right).
Bottom row: real part (left) and imaginary part (right) of the linearised Neumann boundary data $\partial_{\nu'} E_{3} - \partial_{\nu'} E_{3}( ;0)$ on $\partial\Omega_{2}$.}
\label{fig_dEHw15}
\end{figure}

All the above numerical calculation has chosen the same $\xi' = (-1,0)$ satisfying $|\xi'| < k = \omega$. At the same time, it remains to check the numerical performance for those $|\xi'| > 2k$ (or $2 \omega$) whose CE solutions display differently. To save the length of the paper, we only show the linearised Neumann boundary data $\partial_{\nu'} E_{3} - \partial_{\nu'} E_{3}( ;0)$ in Figure \ref{fig_gtdEH} where $\xi' = (-40,0)$ is chosen for both $\omega = 5$ and $15$. It is clear that the linearised Neumann boundary data blows up on parts of the boundary $\partial\Omega_{2}$. It is worth to mention that such behaviour is also observed in the acoustic Helmholtz equation as shown in \cite{ILX2018}.

\begin{figure}[!htb]
\centering
\,\hfill $\omega = 5$ Hz \hfill\,\\
\includegraphics[width=1.0\textwidth]{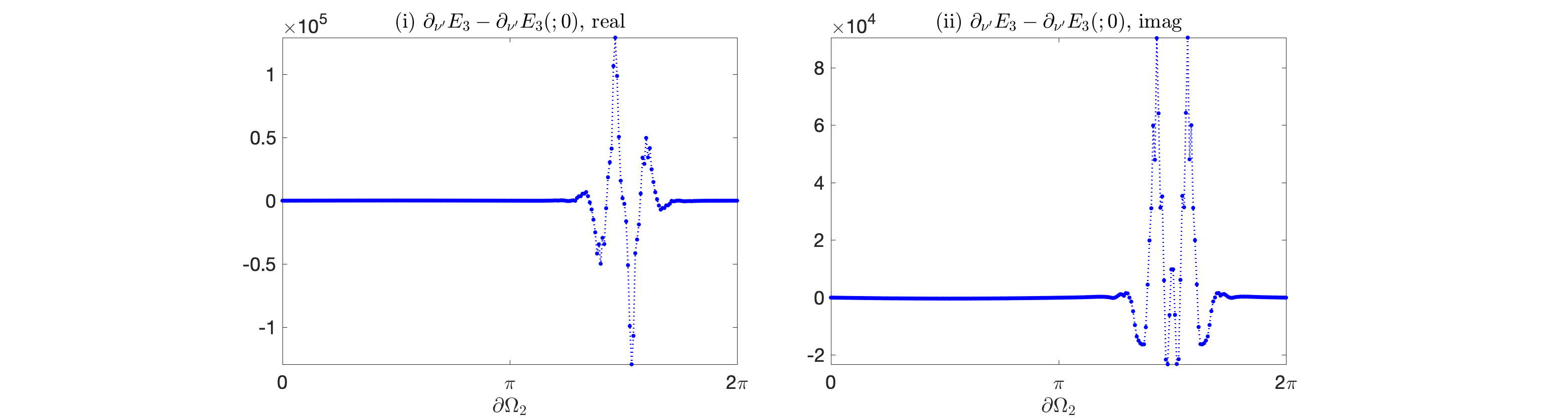}\\
\,\hfill $\omega = 15$ Hz \hfill\,\\
\includegraphics[width=1.0\textwidth]{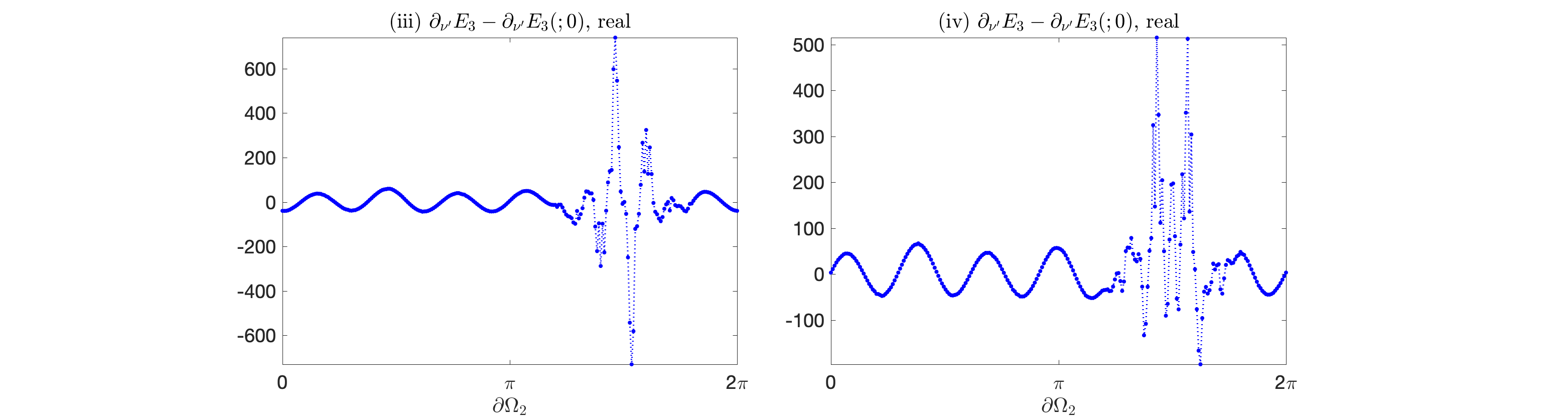}
\caption{[$\xi' = (-40,0)$]
The real part (left) and imaginary part (right) of the linearised Neumann boundary data $\partial_{\nu'} E_{3} - \partial_{\nu'} E_{3}( ;0)$ on $\partial\Omega_{2}$.
Upper row: $\omega = 5$ Hz.
Bottom row: $\omega = 15$ Hz.}
\label{fig_gtdEH}
\end{figure}


\subsection{Inversion of the conductivity and improving resolution under high frequencies}

By numerical calculation of the forward problems and their approximated linearised boundary data in the above subsection we then confirm efficiency of Algorithm 1, especially Steps 8-9 recovering the Fourier modes of the conductivity $\sigma$. Moreover, we verify the improving resolution by numerical evidence when the frequency $\omega$ grows.
To avoid inverse crimes, in the numerical inversion of current subsection, we choose a coarser grid with $90 \times 90$ equal-distance points in the same rectangle domain $[{-0.7},0.7]^{2}$.

To fit the setting of Algorithm 1 to the TE mode and reconstruct the conductivity from the boundary measurement, we recall that a linearised form of the electric field $E_{3}$ is presented by $E_{3} = E_{3}(;0) + E_{3}(;1) + E_{3}(;2)$, where $E_{3}(;0)$, $E_{3}(;1)$ obey subproblems \eqref{E30}, \eqref{E31} and $E_{3}(;2)$ is the remaining higher-order term.
Recalling the equality \eqref{hatsigma2} below
\begin{align*}
- \int_{\Omega_{2}} \sigma E_{3}( ;0) E^{*}_{3}( ;0) \,\mathrm{d}x'
= \frac{1}{\i\omega\mu_{0}} \int_{\partial\Omega_{2}} \partial_{\nu'} E_{3}( ;1) E^{*}_{3}( ;0) \,\mathrm{d}S'
= \int_{\partial\Omega_{2}} \big( \Lambda'_{\textrm{TE}} E_{3}( ;0) \big) E^{*}_{3}( ;0) \,\mathrm{d}S',
\end{align*}
we need to consider the CE solutions $E_{3}( ;0)$, $E^{*}_{3}( ;0)$ referring to those in the proof of Theorem \ref{thm2}. In fact, by choosing any vector $\xi' = (\xi_{1},\xi_{2}) \in \mathbb{R}^{2}$ and $\xi'_{\perp} = (-\xi_{2},\xi_{1})$ such that $\xi' \cdot \xi'_{\perp} = 0$, $|\xi'| = |\xi'_{\perp}| \leq \mathcal{K}$, and denoting $e_{1} = \frac{\xi'}{|\xi'|}$, $e_{2} = \frac{\xi'_{\perp}}{|\xi'|}$, we then generate the CE solutions in \eqref{complexsol2} %
with the complex valued vectors in \eqref{ab2_choice} %
and $k = \omega \sqrt{\epsilon_{0} \mu_{0}}$. %
Thus from the linearised boundary mapping \eqref{hatsigma2} we have
\begin{align}
\label{perturb}
\begin{aligned}
2\pi \,\widehat{\sigma}(-\xi')
&= - \int_{\Omega_{2}} \sigma E_{3}( ;0) E^{*}_{3}( ;0) \,\mathrm{d}x'
= \frac{1}{\i\omega\mu_{0}} \int_{\partial\Omega_{2}} \partial_{\nu'} E_{3}( ;1) E^{*}_{3}( ;0) \,\mathrm{d}S',
\end{aligned}
\end{align}
where the boundary value $\partial_{\nu'} E_{3}( ;1)$ is needed in the right-hand side. Since the electric field $E_{3}( ;1)$ depends on the unknown conductivity $\sigma$ as shown in \eqref{E31}, its Neumann boundary value $\partial_{\nu'} E_{3}( ;1)$ should be approximated by the linearisation
\begin{align*}
\partial_{\nu'} E_{3}( ;1) \approx \partial_{\nu'} E_{3} - \partial_{\nu'} E_{3}( ;0) \quad \textrm{on\ } \partial\Omega_{2},
\end{align*}
with the electric field $E_{3}$ of the original TE mode and $E_{3}( ;0)$ of the unperturbed subproblem.

Referring to the formula \eqref{perturb}, we recover all the Fourier modes $\widehat{\sigma}(\xi')$ of the conductivity $\sigma$ near those $19$ (discrete) inclined segments in the phase space as shown in the middle panel of Figure \ref{fig_sigma} with $\mathcal{K} = 50$. More precisely, Steps 8-9 in Algorithm 1 iterate over all $\xi'$ along each discrete inclined segment containing the phase points $\xi'$ where their modulus $|\xi'|$ are equally distributed in the interval $[0.2,50]$ with a step size $0.2$. The degree of angles between two adjacent inclined segments is $\frac{2}{19}\pi$.

We shall emphasize that the CE solutions generated in Subsection \ref{se_main_TE} have the same properties as in the acoustic wave equation, c.f. \cite[Rem4.1]{ILX2018}. In particular, when the modulus $|\xi'| > 2k$ (or $2\omega$), the CE solutions become (highly) oscillating along one direction and decay exponentially along the vertical direction as partially shown in Figure \ref{fig_gtdEH}. Then approximation of the linearised boundary data is not stable. To visualize the consequence, we present all the recovered absolute Fourier modes $|\widehat{\sigma}(\xi')|$ for the conductivity $\sigma$ at all phase points $\xi'$ satisfying $|\xi'| \leq 50$ in Figure \ref{fig_absFourier} by implementing Algorithm 1. As one can clearly observe, the Fourier modes with the phase modulus $|\xi'| \leq 2k$ (or $2\omega$) can be well recovered. Nevertheless, when the modulus becomes large i.e. $|\xi'| > 2k$ (or $2\omega$), the absolute value of the recovered Fourier modes blows up immediately and one can not expect to include high phase information. Such an observation consists with the conclusions in \cite{ILX2018} where $\mathcal{K} = 2 k$ is the threshold value allowing stable reconstruction, denoting $\mathcal{K}$ be the maximum modulus length of all chosen vectors in the phase space, and $k = \omega$ while $\epsilon_{0} = \mu_{0} = 1$ in our numerical cases.

\begin{figure}[!htb]
\centering
\includegraphics[width=1.0\textwidth]{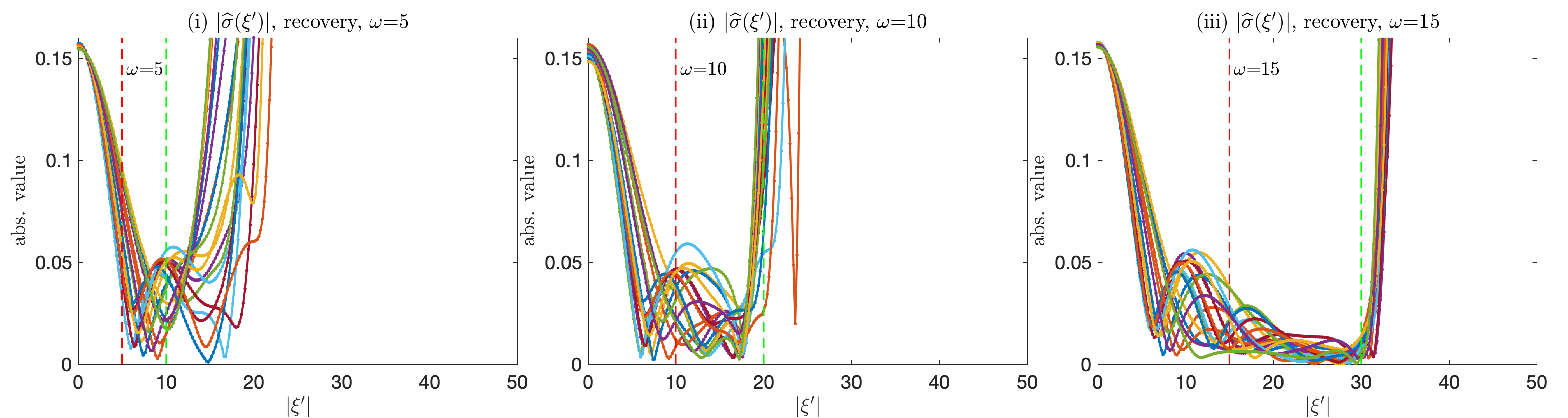}
\caption{Absolute value of recovered Fourier modes with $\omega = 5,10,15$ Hz with $|\xi'| \leq 50$.}
\label{fig_absFourier}
\end{figure}

Then, by utilizing the inverse Fourier transform to the above recovered Fourier coefficients with $\mathcal{K} = 2k$ in Algorithm 1, the reconstructed conductivity $\sigma(x')$ is shown in Figure \ref{fig_sigma_rec} in reference to the exact conductivity in the left panel of Figure \ref{fig_sigma}. As one can observe, in the lower frequency regime $\omega = 5$, no essential information is obtained. On the other hand, if we choose a high frequency regime $\omega = 15$, much higher resolution of the reconstructed conductivity is obtained by using nineteen discrete inclined segments in the phase space. One can further improve the resolution by adding more inclined segments in the phase space and we skip these details.

\begin{figure}[!htb]
\centering
\includegraphics[width=1.0\textwidth]{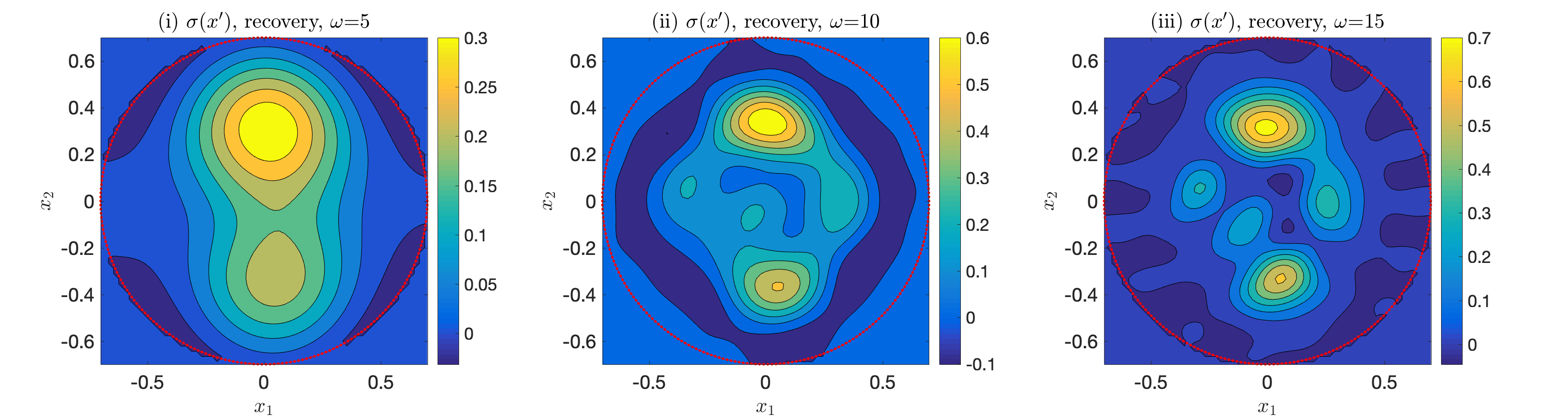}
\caption{Reconstructed conductivity $\sigma$ with $\omega = 5,10,15$ Hz by choosing $\mathcal{K}=2k$ (or $2\omega$).}
\label{fig_sigma_rec}
\end{figure}

Finally, we impose noise propagation on the linearised Neumann boundary data $\partial_{\nu'}E_{3}( ;1) \approx \partial_{\nu'}E_{3} - \partial_{\nu'}E_{3}( ;0)$ on $\partial\Omega_{2}$. Assume that there exists a noise level $\delta$ such that the difference between the exact and noisy boundary data satisfies
\begin{align*}
\left\| \partial_{\nu'}E_{3}^{\delta} - \partial_{\nu'}E_{3} \right\|_{(0)}(\partial\Omega_{2}) \leqslant \delta \left\| \partial_{\nu'}E_{3} \right\|_{(0)}(\partial\Omega_{2}),
\end{align*}
where $\partial_{\nu'}E_{3}^{\delta}$ is the noisy Neumann boundary data.
By choosing different noise levels, we plot the decaying slope of the relative error with respect to the noise level $\delta$ in Figure \ref{fig_noise_slope}. When the frequency is small, for instance $\omega = 5$, one can hardly observe an error bound where the decaying slope is $\delta^{0.00}$. If we increase the frequency to $\omega = 10$ and $\omega = 15$, the decaying slopes grow to $\delta^{0.08}$ and $\delta^{0.20}$.
To have a particular check on the noise propagation towards the inversion resolution, we present the recovered Fourier coefficients and the corresponding reconstructed conductivity in Figure \ref{fig_absFourier_noise} where $0.1$ noise is imposed on the noisy Neumann boundary data with $\omega = 15$. Though the recovered Fourier coefficients become rough when noise appears, the reconstructed conductivity retains good resolution. When the noise increases, the resolution become worse as reflected in Figure \ref{fig_noise_slope}.

\begin{figure}[!htb]
\centering
\includegraphics[width=1.0\textwidth]{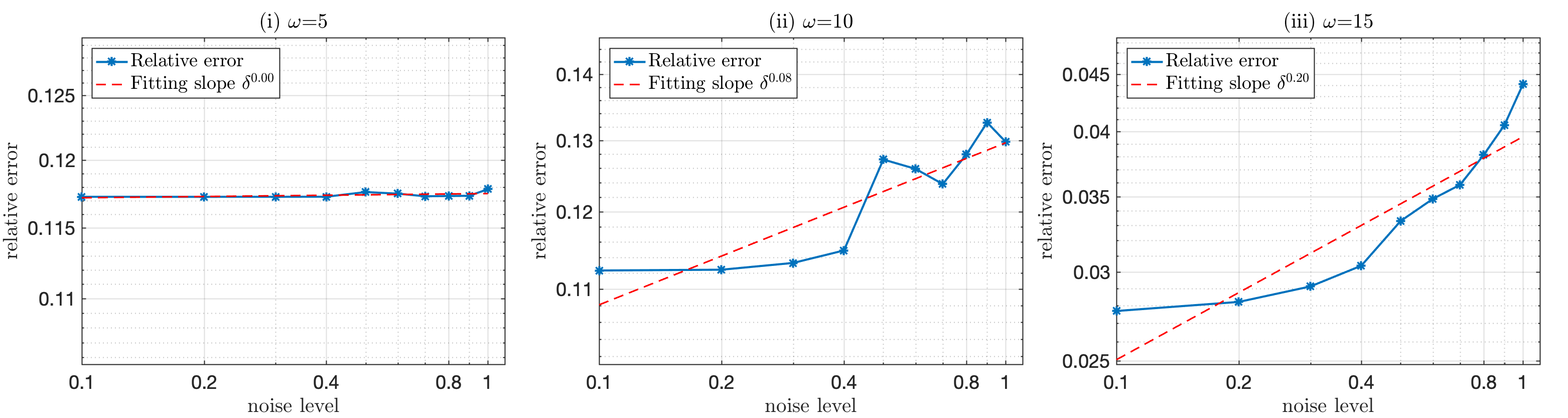}%
\caption{The blue stars are relative error of different noise levels with noisy Neumann boundary data $\partial_{\nu'}E_{3}^{\delta}$. The red dashed line is the fitting line with slopes $\delta^{0.00}$ for $\omega = 5$ Hz, $\delta^{0.08}$ for $\omega = 10$ Hz, $\delta^{0.20}$ for $\omega = 15$ Hz. }
\label{fig_noise_slope}
\end{figure}

\begin{figure}[!htb]
\centering
\includegraphics[width=1.0\textwidth]{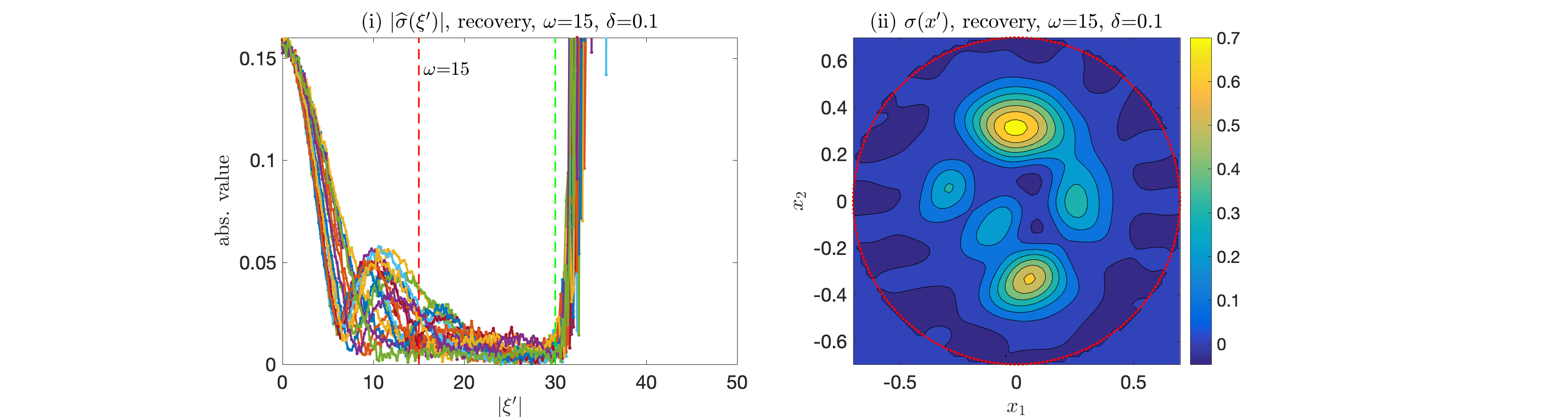}
\caption{[$\omega = 15$ Hz] Absolute value of recovered Fourier coefficients and the corresponding reconstructed conductivity with noise levels $0.1$.}
\label{fig_absFourier_noise}
\end{figure}


\section*{Conclusion}

We partially justified analytically and justified numerically a linearisation approach in the inverse conductivity problem for the Maxwell system at higher frequencies, where the boundary data for the original (non-linear) inverse problem are used in the linearised problem to get a very good approximation of the conductivity coefficient.
To complete an analytic part one expects to demonstrate linearisation in respect to the frequency. Since the numerical resolution is obviously increasing in the linearised version, it would be interesting to handle the original non-linear inverse problem with some use of the analytic methods of \cite{B2008} and already available numerics in \cite{HLSST}.
The next natural step is to solve numerically the inverse problems for the complete three-dimensional Maxwell system and to confirm Theorem \ref{thm1}. However there are substantial difficulties even with the direct problems when generating the boundary data for the inverse problem.
Another important issue is to show that this method is applicable to practical geophysical and medical settings.



\bibliographystyle{siamplain}

\begin{thebibliography}{99}

\bibitem{A1988}
Alessandrini, Giovanni.
{\it Stable determination of conductivity by boundary measurements}.
Applicable Analysis, {\bf 27} (1988), no. 1-3, 153--172.

\bibitem{B2008}
Bukhgeim, Alexander.
{\it Recovering a potential from Cauchy data in the two-dimensional case}.
Journal of Inverse Ill-Posed Problems {\bf 16} (2008), no. 1, 19--33.


\bibitem{BL2005}
Bao, Gang; Li, Peijun.
{\it Inverse medium scattering problems for electromagnetic waves}.
SIAM Journal on Applied Mathematics, {\bf 65} (2005), no. 6, 2049--2066.

\bibitem{BL2009}
Bao, Gang; Li, Peijun.
{\it Numerical solution of an inverse medium scattering problem for Maxwell's equations at fixed frequency}.
Journal of Computational Physics, {\bf 228} (2009), no. 12, 4638--4648.


\bibitem{C2010}
Caro, Pedro.
{\it Stable determination of the electromagnetic coefficients by boundary measurements}.
Inverse Problems {\bf 26} (2010), no. 10, 105014, 25 pp.

\bibitem{C2011}
Caro, Pedro.
{\it On an inverse problem in electromagnetism with local data: stability and uniqueness}.
Inverse Problems \& Imaging, {\bf 5} (2011), no. 2, 297--322.


\bibitem{DS1994}
Dobson, David C.; Santosa, Fadil.
{\it Resolution and stability analysis of an inverse problem in electrical impedance tomography: dependence on the input current patterns}.
SIAM Journal on Applied Mathematics, {\bf 54} (1994), no. 6, 1542--1560.

\bibitem{GFCK2002}
Gibou, Frederic; Fedkiw, Ronald P.; Cheng, Li-Tien; Kang, Myungjoo.
{\it A Second-Order-Accurate Symmetric Discretization of the Poisson Equation on Irregular Domains}.
Journal of Computational Physics {\bf 176} (2002), no. 1, 205--227.

\bibitem{HLSST}
de Hoop, Maarten V.; Lassas, Matti; Santacesaria, Matteo; Siltanen, Samuli; Tamminen, Janne P..
{\it Positive-energy D-bar method for acoustic tomography: a computational study}.
Inverse Problems, {\bf 32} (2016), 025003, 35 pp.

\bibitem{I2011}
Isakov, Victor.
{\it Increasing stability for the Schr\"{o}dinger potential from the Dirichlet-to-Neumann map}.
Discrete \& Continuous Dynamical Systems-S, {\bf 4} (2011), no. 3, 631--640.


\bibitem{ILW2016}
Isakov, Victor; Lai, Ru-Yu; Wang, Jenn-Nan.
{\it Increasing stability for the conductivity and attenuation coefficients}.
SIAM J. Math. Anal. {\bf 48} (2016), no. 1, 569--594.

\bibitem{ILX2018}
Isakov, Victor; Lu, Shuai; Xu, Boxi.
{\it Linearized inverse Schr\"{o}dinger potential problem at a large wavenumber}.
SIAM J. Appl. Math., {\bf 80} (2020), no. 1, 338--358.

\bibitem{IW2014}
Isakov, Victor; Wang, Jenn-Nan.
{\it Increasing stability for determining the potential in the Schr\"{o}dinger equation with attenuation from the Dirichlet-to-Neumann map}.
Inverse Problem \& Imaging, {\bf 8} (2014), no. 4, 1139--1150.




\bibitem{KP2006}
Kreiss, Heinz-Otto; Petersson, N. Anders.
{\it A second order accurate embedded boundary method for the wave equation with Dirichlet data}.
SIAM Journal on Scientific Computing {\bf 27} (2006), no. 4, 1141--1167.

\bibitem{KPY2002}
Kreiss, Heinz-Otto; Petersson, N. Anders; Ystr\:{o}m, Jacob.
{\it Difference approximations for the second order wave equation}.
SIAM Journal on Numerical Analysis {\bf 40} (2002), no. 5, 1940--1967.

\bibitem{M2001}
Mandache, Niculae.
{\it Exponential instability in an inverse problem for the Schr\"{o}dinger equation}.
Inverse Problems, {\bf 17} (2001), 1435--1444.


\bibitem{OPS1993}
Ola, Petri; P\"{a}iv\"{a}rinta, Lassi; Somersalo, Erkki.
{\it An inverse boundary value problem in electrodynamics}.
Duke Mathematical Journal, {\bf 70} (1993), no. 3, 617--653.

\bibitem{R2012}
Rumpf, Raymond C..
{\it Simple implementation of arbitrarily shaped total-field/scattered-field regions in finite-difference frequency-domain}.
Progress In Electromagnetics Research B, {\bf 36} (2012), 221--248.


\bibitem{SIC1992}
Somersalo, Erkki; Isaacson, David; Cheney, Margaret.
{\it A linearized inverse boundary value problem for Maxwell's equations}.
J. Comput. Appl. Math., {\bf 42} (1992), no. 1, 123--136.

\bibitem{SU1987}
Sylvester, John; Uhlmann, Gunther.
{\it Global uniqueness theorem for an inverse boundary value problem}.
Annals of Mathematics, {\bf 125} (1987), no.1, 153--169.


\end{thebibliography}

\end{document}